\documentclass[11pt]{article}
\usepackage[T1]{fontenc}
\usepackage[utf8]{inputenc}
\usepackage{lmodern}
\usepackage{libertine}
\usepackage[a4paper,margin=3cm]{geometry}
\usepackage{microtype}
\usepackage{amsmath,amssymb,amsthm,mathtools}
\usepackage{mathrsfs}
\usepackage{booktabs}
\usepackage{enumitem}
\usepackage[font=small,labelfont=bf]{caption}
\usepackage{xcolor}
\usepackage{tikz}
\definecolor{linkblue}{RGB}{0,70,140}
\usepackage[
  colorlinks=true,
  linkcolor=linkblue,
  citecolor=linkblue,
  urlcolor=linkblue,
  bookmarksopen=true,
  pdfauthor={Enrique Zuazua},
  pdftitle={The Coercivity Gap in Neural PDE Solvers: Parameter Escape and State Convergence}
]{hyperref}
\usepackage[capitalise,noabbrev]{cleveref}
\allowdisplaybreaks
\newtheorem{theorem}{Theorem}[section]
\newtheorem{proposition}[theorem]{Proposition}
\newtheorem{lemma}[theorem]{Lemma}
\newtheorem{corollary}[theorem]{Corollary}
\newtheorem{remark}[theorem]{Remark}

\newcommand{\R}{\mathbb{R}}
\newcommand{\N}{\mathbb{N}}
\newcommand{\M}{\mathcal{M}}
\newcommand{\TV}{\mathrm{TV}}
\newcommand{\supp}{\operatorname{supp}}
\newcommand{\dd}{\,\mathrm{d}}
\newcommand{\ip}[2]{\left\langle #1,#2\right\rangle}
\title{The Coercivity Gap in Neural PDE Solvers:\\
Parameter Escape and State Convergence}

\author{%
  Enrique Zuazua\textsuperscript{[1,2,3]}\\[0.9em]
  \parbox{0.88\textwidth}{%
    \centering\small
    \textbf{[1]} Friedrich--Alexander--Universit\"at Erlangen--N\"urnberg,
    Department of Mathematics, Chair for Dynamics, Control, Machine Learning
    and Numerics (Alexander von Humboldt Professorship), Cauerstr.~11,
    91058 Erlangen, Germany.\\[0.4em]
    \textbf{[2]} Chair of Computational Mathematics, University of Deusto,
    48007 Bilbao, Basque Country, Spain.\\[0.4em]
    \textbf{[3]} Departamento de Matem\'aticas, Universidad Aut\'onoma de Madrid,
    28049 Madrid, Spain.\\[0.6em]
    \texttt{enrique.zuazua@fau.de}%
  }%
}

\date{}

\begin{document}

\maketitle

\begin{abstract}
We study variational PDE solvers based on neural-network ansatzes. In this
setting, the PDE energy is minimized over a nonlinear family of realizations
rather than over a linear trial space. The resulting optimization problem may
have minimizing sequences for which the parameters diverge while the
corresponding states remain bounded and even converge in the natural energy
norm. We call this mismatch the \emph{coercivity gap}.

We first formulate the phenomenon abstractly and then analyze a whole-space
elliptic model with fixed-variance Gaussian mixtures. In this model, the
relevant mechanism is the collision of two active neurons. The realized states
converge to a Gaussian derivative, while the parameters escape and the reduced
minimum is not attained. Under suitable regularity and decay assumptions on
the forcing term, the states still converge in the energy norm, with an
explicit logarithmic rate.

We also discuss measure-valued relaxations, Tikhonov regularization, residual
minimization in stability norms, and a state-space HYCO formulation for hybrid
physics--data problems with bounded observations.
\end{abstract}



\section{Introduction}

\subsection{Problem formulation and main results}
Physics-informed neural networks, the Deep Ritz method, and related variational
or residual formulations are now widely used for partial differential equations
(PDEs); see
\cite{raissi2019pinn,yu2018deep,karniadakis2021pinnreview,mishra2023estimates,
doumeche2025,deryck2024error} and the overview \cite{BBZ}. Their appeal is clear
in high-dimensional and mesh-free settings, but replacing a linear trial space
by a nonlinear realization class changes the
geometry of the approximation problem: classical questions of approximation,
stability, compactness, and convergence acquire new forms, because the topology
that is natural for the physical state need not control the parameters of its
neural representation.

The central phenomenon is the \emph{coercivity gap}. A PDE energy may be
coercive on a function space, but once it is pulled back through a nonlinear
neural parametrization, compactness in parameter space may be lost. In
particular, a sequence of parameter vectors may become unbounded while the
corresponding neural functions remain bounded, or even converge strongly, in
the natural norm.

Figure~\ref{fig:coercivity-gap} illustrates this split. The parameter sequence
may escape, while the realized states remain controlled by the PDE energy. In
that case, the relevant quantity to monitor is the state error, not the
convergence of a particular parameter representation.

This point is important in practice. Parameter escape, large weights, or nearly
colliding centers do not by themselves mean that the PDE approximation has
failed. They may simply reflect an ill-conditioned or nonunique
parametrization, while the computed state is already accurate in the relevant
energy or residual norm. In numerical work, the state, the energy gap, the
residual in a stability norm, and the data misfit are the main objects of
interest; the parameters are only a way to generate them.

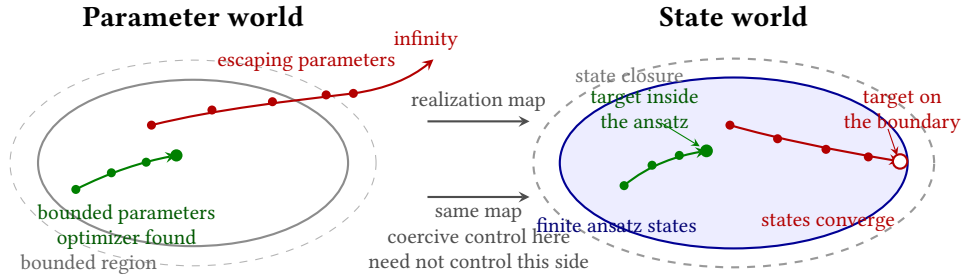
\begin{figure}[ht]
\centering
\begin{tikzpicture}[x=1cm,y=1cm,font=\small,>=stealth]
\begin{scope}
  \node[font=\bfseries] at (0,1.95) {Parameter world};
  \draw[black!45,thick] (0,0) ellipse (2.05 and 1.10);
  \draw[black!30,dashed] (0,0) ellipse (2.42 and 1.36);
  \node[font=\scriptsize,black!60] at (-1.38,-1.34) {bounded region};

  \draw[green!50!black,thick,->]
    (-1.55,-0.35) .. controls (-1.10,-0.12) and (-0.65,0.02) .. (-0.22,0.10);
  \foreach \p in {(-1.55,-0.35),(-1.08,-0.13),(-0.62,0.01)}
    \fill[green!50!black] \p circle (1.7pt);
  \fill[green!50!black] (-0.22,0.10) circle (2.3pt);
  \node[align=center,green!35!black,font=\scriptsize] at (-0.88,-0.84)
    {bounded parameters\\ optimizer found};

  \draw[red!70!black,thick,->]
    (-0.55,0.50) .. controls (0.34,0.72) and (1.30,0.82) .. (2.12,0.92)
    .. controls (2.55,0.98) and (2.88,1.12) .. (3.18,1.37);
  \foreach \p in {(-0.55,0.50),(0.25,0.70),(1.05,0.81),(1.76,0.90),(2.12,0.92)}
    \fill[red!70!black] \p circle (1.7pt);
  \node[align=center,red!70!black,font=\scriptsize] at (1.50,1.34)
    {escaping parameters};
  \node[align=center,red!70!black,font=\scriptsize] at (3.08,1.63)
    {infinity};
\end{scope}

\begin{scope}[xshift=7.15cm]
  \node[font=\bfseries] at (0,1.95) {State world};
  \fill[blue!7] (0,0) ellipse (2.30 and 1.13);
  \draw[blue!58!black,thick] (0,0) ellipse (2.30 and 1.13);
  \draw[black!40,dashed,thick] (0,0) ellipse (2.65 and 1.38);
  \node[blue!45!black,font=\scriptsize] at (-1.55,-0.82)
    {finite ansatz states};
  \node[black!55,font=\scriptsize] at (-1.38,1.16)
    {state closure};

  \draw[green!50!black,thick,->]
    (-1.45,-0.30) .. controls (-1.12,-0.02) and (-0.72,0.10) .. (-0.36,0.16);
  \foreach \p in {(-1.45,-0.30),(-1.08,-0.03),(-0.72,0.10)}
    \fill[green!50!black] \p circle (1.7pt);
  \fill[green!50!black] (-0.36,0.16) circle (2.4pt);
  \node[align=center,green!35!black,font=\scriptsize] at (-1.18,0.72)
    {target inside\\ the ansatz};
  \draw[green!50!black,->] (-0.96,0.53) -- (-0.45,0.22);

  \draw[red!70!black,thick,->]
    (-0.05,0.50) .. controls (0.58,0.32) and (1.32,0.17) .. (2.20,0.02);
  \foreach \p in {(-0.05,0.50),(0.58,0.32),(1.22,0.18),(1.78,0.08)}
    \fill[red!70!black] \p circle (1.7pt);
  \draw[red!70!black,fill=white,thick] (2.20,0.02) circle (2.7pt);
  \node[align=center,red!70!black,font=\scriptsize] at (1.25,-0.78)
    {states converge};
  \node[align=center,red!70!black,font=\scriptsize] at (2.23,0.68)
    {target on\\ the boundary};
  \draw[red!70!black,->] (2.04,0.45) -- (2.16,0.10);
\end{scope}

\draw[->,thick,black!70] (3.10,0.55) -- (4.45,0.55)
  node[midway,above,font=\scriptsize] {realization map};
\draw[->,thick,black!70] (3.10,-0.45) -- (4.45,-0.45)
  node[midway,below,font=\scriptsize] {same map};
\node[align=center,font=\scriptsize,black!70] at (3.78,-1.13)
  {coercive control here\\ need not control this side};
\end{tikzpicture}
\caption{The coercivity gap. The algorithm searches in parameter space, while
the PDE energy controls the realized state. When the target is represented
inside the finite ansatz, a minimizing sequence may stay bounded. When the
target is reached only as a boundary state, the realized states may converge
even though the corresponding parameters escape.}
\label{fig:coercivity-gap}
\end{figure}

This distinction can be hidden by the way the discrete problem is formulated.
Some analyses introduce an auxiliary constrained or regularized parameter
problem: the weights are restricted to a bounded set, compactness is imposed
a priori, or a direct parameter penalty is added. Such devices are useful for
proving existence or convergence, but they should not be confused with
coercivity of the original PDE loss pulled back through the neural
parametrization. By imposing compactness at the parameter level, they rule out
parameter escape by assumption. The coercivity gap studied here concerns the
untruncated realization problem, where the PDE energy controls the realized
state but need not control the parameters that represent it.

We first formulate this phenomenon abstractly, and then study it for a
whole-space elliptic equation with fixed-variance Gaussian mixtures. In that
model the neural states converge quantitatively, while the reduced optimization
problem has no minimizer and no bounded minimizing parameter sequence.

The main contributions of this paper are the following.
\begin{enumerate}[label=(\roman*),leftmargin=*]
\item a Hilbert-space abstract description of the coercivity gap phenomenon;
\item a whole-space, fixed-variance Gaussian construction in \(H^1(\R^d)\):
      two distinct active neurons collide and converge to a Gaussian derivative
      outside every finite realization class; a suitable elliptic right-hand
      side then yields an unattained reduced infimum and forces every
      parameter-space minimizing sequence to escape;
\item quantitative state estimates, including a constructive whole-space
      \(O((\log P)^{-1/2})\) bound under weighted regularity, with the neuron
      count made explicit; and
\item extensions to measure relaxation, Tikhonov regularization, residual
      minimization, and a HYCO-type hybrid formulation, always in norms
      compatible with PDE stability.
\end{enumerate}

We use fixed-variance Gaussian radial-basis functions
(see \cite{buhmann2000,madych1992,wendland2004}) because the collision
mechanism, the closure defect, and the approximation estimates can be written
explicitly. No claim of architectural optimality is intended. The point is
rather to show, in a model where the mechanism can be followed by hand, that
parameter escape can coexist with convergence in the PDE energy norm.

Most of the ingredients are known in approximation theory and neural-network
topology; see the bibliographical discussion below. The point here is to place
them inside a variational, or residual, PDE setting where nonattainment and
parameter escape occur together with strong state convergence.

The estimates also clarify why good state-level performance can coexist with
poor parameter behavior. This gives a mathematical reason for a common
computational observation: neural PDE solvers may produce reliable fields even
when their weights, centers, or other coordinates are hard to interpret.

The main point is a state-space one. A neural PDE solver maps parameters to a
physical field, and the PDE energy, residual, or observation loss depends on
that field. Divergence or nonuniqueness of the coordinates is therefore a
conditioning or identifiability issue, not a PDE error by itself. Conversely,
a small training loss is meaningful for the PDE only when it is measured in a
norm that controls the state.

\subsection{Literature review}

The parameter--realization distinction has been analyzed in the context of the
existence of best neural approximants \cite{girosi1990}. The
instability of best-approximation maps, including for Gaussian radial-basis
networks, was studied in \cite{kainen1999}. For fixed architectures,
\cite{petersen2021} analyzed the nonclosedness of realization sets for broad
classes of activations, the failure of inverse stability, and the resulting
blow-up of weights along near-best approximation sequences when a best
approximation is not attained. In \cite{mahan2021}, nonclosedness phenomena
were extended to Sobolev topologies by means of two-neuron difference quotients
of the form
\begin{equation}\label{collapse}
 n\bigl[\rho(\,\cdot+1/n)-\rho(\,\cdot)\bigr]\longrightarrow \rho',
\end{equation}
which provide examples of Sobolev convergence accompanied by unbounded
parameter growth. This construction is the direct approximation-theoretic
precursor of the collision mechanism used below. Related questions about
closedness and existence of optimal parameters, in particular for sparse ReLU
architectures, are studied in \cite{le2023}. 

Here nonclosedness, nonattainment, and parameter escape are tied to a coercive
elliptic energy and to the approximation of a PDE solution. At a broader
analytical level, cancellation between nearby modes is also a classical device
for constructing localized wave packets \cite{zuazua2005waves}.

We also use the term \emph{neuron condensation} in a narrower sense. In the
training-dynamics literature, it usually refers to a regime in which many
neurons concentrate along a small number of feature-space directions
\cite{luo2021phase,zhou2022condensation}. The mechanism here is different.
Two or more fixed-variance Gaussian centers coalesce, their signed amplitudes
diverge, and the realized functions converge to a Gaussian derivative, as in
\eqref{collapse}.

A complementary line of work establishes convergence at the level of the
realized physical states. For PINNs and residual minimization methods, several
results prove convergence under suitable assumptions on PDE stability,
regularity, sampling, approximation, and optimization
\cite{shin2020,doumeche2025,deryck2024error}. For the Deep Ritz method,
\(\Gamma\)-convergence and equicoercivity arguments similarly imply convergence
of realizations of global quasi-minimizers
\cite{muller2019deepritz,dondl2022}. Coercivity and compactness have also been
used as a stability framework for neural PDE approximations
\cite{gazoulis2023}.

These works establish convergence at the state level. Our point is different.
The realized states may converge strongly while the parameters representing them
are noncompact or divergent. This is the parameter--state separation that the
present paper isolates.

The closest PDE precedent for this parameter--state separation is
\cite{bertoluzza2024}, in the setting of weak adversarial networks (WANs). There
the authors observe that neural trial classes need not be closed, analyze weakly
convergent minimizing sequences, derive quasi-best-approximation bounds for the
limiting states, and note that parameter norms must become unbounded when the
limit does not belong to the neural class. Their work is therefore close in
spirit to the mechanism studied here. The present paper formulates the issue as
a variational gap between coercivity in state space and compactness in parameter
space.

A separate source of ill-posedness arises when a continuous variational or
residual functional is replaced by finitely many measurements. Recent work
shows that quadrature-based Deep Ritz losses and weak PINN formulations can
then possess nonunique or degenerate minimizers \cite{langer2026}. This
finite-information mechanism is distinct from ours: the Gaussian collision
below already causes nonattainment for the exact continuous coercive energy,
before sampling or quadrature is introduced.

Measure-valued relaxed formulations give a convex way of looking at neural
approximation \cite{bach2017,chizat2018,savarese2019}. The same issue persists
there: unless the total variation of the representing measures is penalized,
the measures may diverge while the states still converge.
Indeed, the Gaussian relaxation considered below gives an explicit example of this limitation: the finite-dimensional collision disappears as a parametrization issue, but a related lack of coercivity may persist at the level of representing measures.

A related but distinct issue appears in finite element approximations of sharp
Sobolev constants \cite{ignatzuazua}. There the approximation spaces are linear, but the
variational problem has a nonlinear manifold of minimizers, and small energy
deficit must be converted into proximity to this manifold, modulo the relevant
symmetries. This is different from the parameter--state coercivity gap studied
here, but it reflects the same general principle that variational convergence
depends on identifying the correct object of stability.

The contribution is to place this nonclosedness mechanism inside a coercive PDE
energy and to quantify the resulting state convergence. This separates the
physical stability estimate from the compactness, or lack of compactness, of
the chosen neural parametrization.

\subsection{Structure of the paper}
The paper proceeds as follows. \Cref{sec:abstract} gives the abstract stability
estimates. \Cref{sec:model} introduces the Gaussian elliptic model.
\Cref{sec:escape} proves the collision, nonclosedness, and nonattainment
results. State convergence and rates are proved in
\Cref{sec:state-convergence}. Relaxation,
regularization, kernel choice, residual losses, and HYCO variants are discussed
in \Cref{sec:relaxation,sec:alternative-kernels,sec:residual-hybrid}. The
appendices give the weighted spectral and localization estimates.


\section{The abstract coercivity gap}
\label{sec:abstract}

\subsection{An abstract setting}

Let \(X\) be the physical state space, let \(E:X\to\R\) be a convex energy, and
let the nonlinear ansatz classes be generated by realization maps. For each
\(P\in\N\), we consider a parameter set \(\Theta_P\), a map
\[
\Phi_P:\Theta_P\longrightarrow X,
\]
and the realized class \(\mathcal A_P:=\Phi_P(\Theta_P)\subset X\).
Thus, the relevant objects are organized at three distinct levels:
\[
\text{parameters}
\quad\longrightarrow\quad
\text{realized states}
\quad\longrightarrow\quad
\text{energy values},
\]
or, more precisely,
\begin{equation}
\label{eq:three-levels}
\Theta_P
\xrightarrow{\ \Phi_P\ }
\mathcal A_P:=\Phi_P(\Theta_P)\subset X
\xrightarrow{\ E\ }
\R,
\qquad
\mathcal L_P:=E\circ\Phi_P.
\end{equation}

The coercivity gap appears precisely in this three-level structure.
Coercivity of the physical energy \(E\) controls the realized state
\(\Phi_P(\Theta)\) in \(X\), but it does not by itself control the parameter
vector \(\Theta\). Thus, bounded or even vanishing energy error may imply
compactness and convergence of the states, while the corresponding parameters
may have no bounded subsequence. The results below separate these two effects:
the abstract stability theorem explains why near-minimal energies imply
convergence of realized states, whereas the nonattainment and collision
results show how the associated minimizing parameters can escape.

The separation is therefore useful, not merely negative. It tells us that the
right compactness question for a PDE solver is first a compactness question for
the realized states. Parameter compactness is relevant for conditioning,
attainment, and algorithmic robustness, but it is not the same as convergence
of the numerical solution in the PDE topology. In particular, parameter escape
should be read as a warning about representation and training, not as automatic
evidence that the computed state is inaccurate.

More precisely, the following three questions concern different aspects of the problem and
should therefore be treated separately:
\begin{enumerate}[label=(\roman*),leftmargin=*]
\item \emph{State stability:} do near-minimal reduced losses imply that
      \(\Phi_P(\Theta)\) is close to the exact minimizer in \(X\)?
\item \emph{Attainment and parameter compactness:} is the infimum over
      \(\mathcal A_P\) attained, and can minimizing parameters remain bounded?
\item \emph{Training:} can a given algorithm reach such near-minimal losses,
      and at what cost?
\end{enumerate}

The next theorem addresses the first question. Under strong convexity, energy
accuracy gives strong convergence of realized states. It says nothing about
attainment or boundedness of parameters. In \Cref{sec:escape}, the Gaussian
collision answers the second question negatively for a concrete elliptic
energy: the reduced infimum is not attained and every minimizing parameter
sequence escapes.

Training dynamics are a separate issue. Our results start from global, or
sufficiently near-global, minimizing sequences and analyze the behavior of
their realized states in the PDE energy space. They do not explain how such
sequences are produced by a particular optimizer. That question depends on the
chosen algorithm, initialization, step sizes, regularization, and the nonconvex
loss landscape. We therefore separate the variational mechanism of state
convergence and parameter escape from the algorithmic problem of reaching the
relevant near-minimal energy levels.

We assume no closedness, injectivity, or compactness of the parametrization.
The only structural condition is that the realized classes are nested,
\[
\mathcal A_P\subset \mathcal A_{P+1},
\]
as is the case, for instance, when an additional zero-weight unit can be added.

The additional assumption used below is strong convexity, which converts an
energy error into a norm error for the realized states. Recall that
\(E:X\to\R\) is \(\lambda\)-strongly convex if, for every \(v,w\in X\) and
\(t\in[0,1]\),
\[
E((1-t)v+tw)
\le (1-t)E(v)+tE(w)
-\frac{\lambda}{2}t(1-t)\|v-w\|_X^2.
\]

The state-space argument has the usual structure of a
\(\Gamma\)-convergence proof; see Dal Maso \cite{dalmaso1993}. At that level
one combines a liminf, or stability, mechanism with a recovery mechanism
showing that the approximating classes envelop the limiting problem. In
Galerkin methods, finite-dimensional approximation properties provide the
recovery step; in neural methods, the analogous role is played by universal
approximation, or more generally by density of the realization classes.

This variational picture is clear as long as it is formulated in the state
space. It becomes less clear after pulling the problem back to parameters.
\(\Gamma\)-convergence and equicoercivity in the physical topology may imply
compactness and convergence of realized states, but they do not by themselves
give compactness of parameter representatives. The coercivity gap is precisely
this missing implication: the states may converge in the PDE energy space while
the parameters used by the numerical method escape.

Parameter constraints and regularizers remove this escape mechanism by
assumption. They are often useful, but they should be distinguished from
coercivity of the physical loss itself.

\subsection{Energy accuracy implies state accuracy}

Two errors enter the estimate. The best energy value over \(\mathcal A_P\)
measures the expressive accuracy of the class, whether or not that value is
attained. A computed state generally reaches it only up to an optimization
tolerance. The theorem turns these two quantities into a state-error bound.

\begin{theorem}[State stability for nonclosed realization classes]
\label{thm:abstract-state-stability}
Let \(X\) be a Hilbert space. Let \(E:X\to\R\) be coercive, weakly lower
semicontinuous, and
\(\lambda\)-strongly convex for some \(\lambda>0\). Let \(u\in X\) denote its unique
minimizer and define
\[
m_P:=\inf_{v\in\mathcal A_P}E(v),
\qquad
\alpha_P:=m_P-E(u)\ge0.
\]
Assume that there are \(a_P\in\mathcal A_P\) such that
\[
a_P\to u\quad\text{in }X,
\qquad
E(a_P)\to E(u).
\]
Then:

\begin{enumerate}[label=(\alph*),leftmargin=*]
\item The best energy errors satisfy
\begin{equation}
\label{eq:abstract-best-energy-error}
0\le\alpha_P\le E(a_P)-E(u)\longrightarrow0.
\end{equation}

\item If \(v_P\in\mathcal A_P\) satisfies
\[
E(v_P)\le m_P+\varepsilon_P,
\qquad \varepsilon_P\ge0,
\]
then
\begin{equation}
\label{eq:abstract-state-error-budget}
\|v_P-u\|_X^2
\le \frac{2}{\lambda}\bigl(\alpha_P+\varepsilon_P\bigr).
\end{equation}
In particular, \(v_P\to u\) strongly in \(X\) whenever
\(\varepsilon_P\to0\).

\item For fixed \(P\), every minimizing sequence
\((v_P^k)_k\subset\mathcal A_P\) is bounded in \(X\). If \(\bar v_P\) is any
of its weak accumulation points, then
\begin{equation}
\label{eq:abstract-relaxed-state-bound}
E(\bar v_P)\le m_P,
\qquad
\|\bar v_P-u\|_X^2\le\frac{2}{\lambda}\alpha_P.
\end{equation}
Thus every selection of fixed-\(P\) relaxed accumulation states converges
strongly to \(u\) as \(P\to\infty\).
\end{enumerate}
These conclusions do not require bounded or convergent parameters, or
attainment of \(m_P\) in \(\mathcal A_P\).
\end{theorem}

\begin{proof}
Strong convexity and minimality of \(u\) imply
\begin{equation}
\label{eq:strong-convex-lower-bound}
E(v)-E(u)\ge \frac{\lambda}{2}\|v-u\|_X^2,
\qquad v\in X.
\end{equation}
Indeed, apply strong convexity to \((1-t)u+tv\), use the minimality of \(u\),
and let \(t\downarrow0\).

Since \(E(u)\le m_P\le E(a_P)\), the approximation hypothesis gives
\eqref{eq:abstract-best-energy-error}. For the near-minimizers in part (b),
\[
\frac{\lambda}{2}\|v_P-u\|_X^2
\le E(v_P)-E(u)
\le \alpha_P+\varepsilon_P,
\]
which proves \eqref{eq:abstract-state-error-budget}.

For part (c), coercivity makes each fixed-width minimizing sequence bounded.
Hilbert-space weak compactness then provides a weakly convergent subsequence.
Weak lower semicontinuity gives
\(E(\bar v_P)\le m_P\). Consequently,
\[
\frac{\lambda}{2}\|\bar v_P-u\|_X^2
\le E(\bar v_P)-E(u)
\le\alpha_P,
\]
which proves \eqref{eq:abstract-relaxed-state-bound}.
\end{proof}

\begin{remark}[The state-error estimate and the scope of the theorem]
\label{rem:abstract-error-budget}
Estimate \eqref{eq:abstract-state-error-budget} should be read as a
state-space estimate with the following features:
\begin{itemize}
    \item It depends only on the class approximation error \(\alpha_P\)
    and on the optimization tolerance \(\varepsilon_P\).
    \item It contains no bound on the parameter norm and no conditioning
    estimate for the neural parametrization.
    \item It assumes access to a global near-minimal energy level. Thus, it
    does not address the algorithmic question of how such a near-minimizer is
    found.
    \item In particular, the result does not make the pulled-back loss
    \(\mathcal L_P\) convex in the parameters and does not analyze gradient
    descent.
    \item Strong convexity is used only at the level of the physical energy.
    Its role is to convert a small energy gap into a direct norm estimate for
    the realized state.
    \item More general \(\Gamma\)-convergence arguments can yield qualitative
    convergence under weaker assumptions. The quadratic elliptic energy
    considered below provides the sharper identity needed for the explicit
    error bound.
\end{itemize}
\end{remark}

\subsection{Nonattainment and parameter escape}

\Cref{thm:abstract-state-stability} remains valid when \(\mathcal A_P\) is
nonclosed. A fixed-width minimizing sequence then has weak accumulation points
in the weak closure of \(\mathcal A_P\), but those relaxed states need not
belong to \(\mathcal A_P\). In the Gaussian collision below the convergence is
strong, and the limit lies in
\(\overline{\mathcal A_P}^{\,X}\setminus\mathcal A_P\). The PDE energy still
controls the states, while the reduced parameter problem may have no minimizer.
In finite dimensions, this forces parameter escape.

\begin{proposition}[Nonattainment forces parameter escape]
\label{prop:abstract-parameter-escape}
Fix \(P\), let \(\Theta_P=\R^{N_P}\), and suppose that
\(\Phi_P:\Theta_P\to X\) and \(E:X\to\R\) are continuous. If
\[
\inf_{\Theta\in\Theta_P}E(\Phi_P(\Theta))=m_P
\]
is not attained, then every parameter-space minimizing sequence
\((\Theta_k)_k\) satisfies \(\|\Theta_k\|\to\infty\).
\end{proposition}

\begin{proof}
If \(\|\Theta_k\|\not\to\infty\), a subsequence lies in a bounded subset of
\(\R^{N_P}\). It has a further convergent subsequence,
\(\Theta_{k_\ell}\to\Theta_*\). Continuity then yields
\[
E(\Phi_P(\Theta_*))
=\lim_{\ell\to\infty}E(\Phi_P(\Theta_{k_\ell}))=m_P,
\]
contradicting nonattainment.
\end{proof}

The converse is false: parameters can diverge because of redundant
representations even when a minimizer exists. This is why
\Cref{sec:escape} removes inactive-neuron, permutation, and repeated-center
explanations and constructs a collision of two distinct active neurons. Its
nonattainment result is the architecture-specific complement to the
architecture-independent state theorem above.

The abstract framework will be instantiated in the next sections by a
whole-space elliptic equation and a fixed-variance Gaussian realization class.
There, \(X=H^1(\R^d)\), the energy is the elliptic functional \(J\), and the
realization map sends weights and centers to Gaussian mixtures. In that model,
the identity
\[
J(v)-J(u)=\frac12\|v-u\|_{H^1}^2
\]
makes the state stability mechanism completely explicit, while the Gaussian
collision construction shows that the corresponding parameter problem may have
no minimizer and that every minimizing parameter sequence escapes.

\subsection{Residual losses}

The same principle underlying the coercivity gap also appears in residual formulations. What is needed for state convergence is a lower stability estimate for the residual in the norm in which the PDE error is measured. After this state-space estimate has been fixed, approximation and optimization errors enter the error bound, but no corresponding coercive control of the neural parameters is implied.

\begin{proposition}[Abstract stable-residual estimate]
\label{prop:abstract-residual}
Let \(X\) and \(Y\) be Hilbert spaces, let \(B:X\to Y\) be bounded, and assume
that
\begin{equation}
\label{eq:abstract-stability}
\|z\|_X\le C_B\|Bz\|_Y,
\qquad z\in X.
\end{equation}
For \(g=Bu\), define
\[
\mathscr R_B(v):=\|Bv-g\|_Y^2,
\qquad
d_P:=\inf_{w\in\mathcal A_P}\|w-u\|_X.
\]
If \(v_P\in\mathcal A_P\) satisfies
\[
\mathscr R_B(v_P)
\le\inf_{w\in\mathcal A_P}\mathscr R_B(w)+\eta_P,
\qquad \eta_P\ge0,
\]
then
\begin{equation}
\label{eq:abstract-residual-error}
\|v_P-u\|_X
\le C_B\bigl(\|B\|d_P+\sqrt{\eta_P}\bigr).
\end{equation}
In particular, \(v_P\to u\) in \(X\) whenever \(d_P\to0\) and
\(\eta_P\to0\), irrespective of parameter compactness.
\end{proposition}

\begin{proof}
By stability and near-minimality,
\begin{align*}
\|v_P-u\|_X
&\le C_B\mathscr R_B(v_P)^{1/2}\\
&\le C_B
\left(
\inf_{w\in\mathcal A_P}\|B(w-u)\|_Y^2+\eta_P
\right)^{1/2}\\
&\le C_B\bigl(\|B\|d_P+\sqrt{\eta_P}\bigr).
\end{align*}
\end{proof}

\begin{remark}[Choice of residual norm]
The residual must be measured in a norm that controls the state error. The
essential PDE input is a stability estimate of the form
\[
\|v-u\|_X \le C\|Bv-f\|_Y.
\]
Thus the residual space \(Y\) is not arbitrary: it has to be chosen so that
this estimate is valid.

For a second-order elliptic problem, an \(H^{-1}\) residual naturally controls
the error in \(H^1\). By contrast, an \(L^2\) residual is tied to an
\(H^2\)-regularity framework. It controls an \(H^2\)-type error only when the
trial functions belong to the domain of the operator as a map into \(L^2\), and
when the corresponding elliptic regularity estimate holds for the domain,
coefficients, and boundary conditions under consideration.

If \(B\) has a nontrivial kernel, the residual controls only the distance to
the solution set. Boundary conditions, normalization constraints, or other side
conditions are then needed to recover uniqueness. For nonlinear operators, the
same principle applies, but the stability estimate is usually conditional: it
must hold on the part of the state space reached by the minimizing sequence.

If boundary or initial conditions are imposed through the loss, they must be
included in the residual norm, leading to a product-space formulation. These
are the alternatives developed in \Cref{sec:residual-hybrid}.
\end{remark}

\section{An elliptic model and a Gaussian realization class}
\label{sec:model}

\subsection{A coercive elliptic energy and nonclosed Gaussian classes}
\label{subsec:elliptic-gaussian-class}

We consider
\begin{equation}
\label{eq:elliptic}
-\Delta u+u=f
\qquad\text{in }\R^d,
\end{equation}
where \(f\in H^{-1}(\R^d)\). We use the energy norm
\[
\|v\|_{H^1}^2
:=\int_{\R^d}\bigl(|\nabla v|^2+|v|^2\bigr)\dd x
\]
and its associated dual norm. The weak solution \(u\in H^1(\R^d)\) is the
unique minimizer of
\begin{equation}
\label{eq:energy}
J(v):=\frac12\|v\|_{H^1}^2-\ip{f}{v}_{H^{-1},H^1}.
\end{equation}
Since \(u\) solves \eqref{eq:elliptic},
\begin{equation}
\label{eq:energy-identity}
J(v)-J(u)=\frac12\|v-u\|_{H^1}^2,
\qquad v\in H^1(\R^d).
\end{equation}
Thus \(J\) is coercive and \(1\)-strongly convex.

When \(J\) is restricted to a closed finite-dimensional linear Galerkin space,
these properties ensure a unique discrete minimizer and imply the usual
best-approximation stability estimate. The issue studied here arises because a
nonlinear realization class need not be linear or closed: coercivity of
\(J\) controls the state norm but does not make parameter sublevels compact.

Let
\begin{equation}
\label{eq:gaussian}
G(x):=(4\pi)^{-d/2}e^{-|x|^2/4},
\qquad \widehat G(\xi)=e^{-|\xi|^2},
\end{equation}
where \(\widehat v(\xi)=\int_{\R^d}e^{-ix\cdot\xi}v(x)\dd x\). For
\(P\ge1\), define the realization map
\begin{equation}
\label{eq:realization-map}
\Phi_P(\Theta)(x)
:=\sum_{j=1}^P w_jG(x-x_j),
\qquad
\Theta=((w_j,x_j))_{j=1}^P\in\R^{P(d+1)},
\end{equation}
and its image
\[
\M_P:=\Phi_P(\R^{P(d+1)})\subset H^1(\R^d).
\]
Although the parameter space has dimension \(P(d+1)\), \(\M_P\) is not a
smooth manifold globally: permutations, vanishing weights, repeated centers,
and weight splitting create singular strata and nonunique representations.
We therefore call \(\M_P\) a \emph{realization class}.

For example, if two centers coincide, then
\[
aG(\cdot-x)+bG(\cdot-x)=(a+b)G(\cdot-x),
\]
so the two parameter vectors \(((a,x),(b,x))\) and
\(((a+b,x),(0,y))\) represent the same element of \(\M_P\), for any
\(y\in\R^d\). Similarly, permuting the neurons or changing the center of a
zero-weight neuron does not change the realized function.

Set
\begin{equation}
\label{eq:reduced-infimum}
I:=J(u),
\qquad
I_P:=\inf_{v\in\M_P}J(v)
=\inf_{\Theta\in\R^{P(d+1)}}J(\Phi_P(\Theta)).
\end{equation}
The classes are nested because zero-weight neurons may be added. Finite linear
combinations of translates of \(G\) are dense in \(H^1(\R^d)\), a consequence
of Wiener's Tauberian theorem \cite{wiener1932}. Hence
\begin{equation}
\label{eq:minima-converge}
I_P\downarrow I.
\end{equation}
The reduced infimum \(I_P\) need not, however, be attained.

\begin{remark}[Fixed variance]
In the present ansatz, only the centers and output weights are allowed to vary;
the Gaussian variance is kept fixed. This restriction is intentional. It gives
a simple realization class in which the collision mechanism can be isolated
without additional degeneracies.

The fixed variance may, of course, be inefficient when the target function has
a very different characteristic scale. Allowing trainable variances would
enlarge the approximation class and could improve efficiency. At the same time,
it would introduce further noncompactness mechanisms, such as dilation or
concentration of the Gaussian profiles. These scale-related degeneracies are
not analyzed here.
\end{remark}

\subsection{Dictionary with the abstract coercivity-gap framework}
\label{subsec:gaussian-dictionary}

The abstract objects of \Cref{sec:abstract} specialize in the present Gaussian
elliptic model as follows:
\begin{center}
\begin{tabular}{@{}ll@{}}
\toprule
\textbf{Abstract object} & \textbf{Gaussian elliptic model}\\
\midrule
\(X\) & \(H^1(\R^d)\)\\
\(E\) & the elliptic energy \(J\) in \eqref{eq:energy}\\
\(u\) & the solution of \(-\Delta u+u=f\)\\
\(\Theta_P\) & \(P\) weights and \(P\) centers in \(\R^{P(d+1)}\)\\
\(\mathcal A_P\) & the Gaussian realization class \(\M_P\)\\
\(m_P\) & the reduced infimum \(I_P\)\\
\(\alpha_P\) & \(I_P-I\), the best energy error\\
\bottomrule
\end{tabular}
\end{center}
For this concrete problem, the strong-convexity estimate of
\Cref{thm:abstract-state-stability} becomes the exact identity \eqref{eq:energy-identity}.
Thus the abstract coercivity gap takes a particularly transparent form: the
elliptic energy controls convergence of the realized Gaussian states in
\(H^1(\R^d)\), while the reduced parameter problem may fail to attain its
infimum and its minimizing parameter sequences may escape.

\subsection{Comparison with Galerkin approximations}

The contrast with a classical Galerkin discretization can be stated exactly
for the quadratic model. Let \(V\subset H^1(\R^d)\) be a closed linear trial
space. The minimizer \(u_V\in V\) is characterized by
\begin{equation}
\label{eq:galerkin-orthogonality}
(u-u_V,v)_{H^1}=0,
\qquad v\in V,
\end{equation}
and is therefore the \(H^1\)-orthogonal projection of \(u\) onto \(V\).
Consequently,
\begin{equation}
\label{eq:galerkin-best-approximation}
\|u-u_V\|_{H^1}
=\inf_{v\in V}\|u-v\|_{H^1}.
\end{equation}
The discrete problem has a unique state minimizer, even though a chosen basis
may still introduce an algebraic conditioning issue.

For a general realization class \(\mathcal A\subset H^1\), linearity is not
needed to identify the infimum. Identity \eqref{eq:energy-identity} gives
\begin{equation}
\label{eq:arbitrary-class-infimum}
\inf_{v\in\mathcal A}J(v)-J(u)
=\frac12\operatorname{dist}_{H^1}(u,\mathcal A)^2.
\end{equation}
The projection mechanism is lost: a best approximant need not exist, need not be
unique, and no Galerkin orthogonality is available. In the Gaussian example of
\Cref{sec:escape}, the distance is zero although the target does not belong to
the class, so the reduced energy infimum equals the exact minimum but is not
attained. Approximation accuracy, state well-posedness, parameter
identifiability, and conditioning are therefore separate questions.

\begin{remark}[Fixed centers and random features]
If the centers \(x_1,\ldots,x_P\) are fixed in advance, then the ansatz reduces
to the linear space
\[
V_P(x_1,\ldots,x_P)
=\operatorname{span}\{G(\,\cdot-x_j):1\le j\le P\}.
\]
This includes, for instance, random-feature-type constructions, where the
centers are sampled and then kept fixed. In that setting the only trainable
variables are the output weights, and the minimization problem is linear
quadratic. The weights are determined by a positive-definite Gram system,
provided the fixed centers are distinct.

This fixed-center problem can still be numerically ill-conditioned: if two
centers are very close, the corresponding Gaussian translates are nearly
linearly dependent. However, this is a conditioning issue inside a fixed linear
space, not the nonclosedness mechanism studied here.

The coercivity gap considered in this paper appears when the centers are also
trained. Then the realization class itself can approach collision
configurations, and the reduced infimum may fail to be attained. Thus training
the centers adds adaptivity, but it also opens the door to parameter escape.

A more quantitative analysis of this mechanism, including the dependence of
the Gram matrices on the centers and their degeneration near collision
configurations, is being developed in the work in preparation \cite{daniel}.
\end{remark}
\section{Collision, nonclosedness, and nonattainment}
\label{sec:escape}

The following collision is not a mere redundancy of the parametrization. The
two neurons remain active and distinct for every \(h>0\), but their realized
states converge to a limit outside the finite Gaussian realization class.
Thus the relevant geometric fact is nonclosedness in the state topology:
\(\partial_1G\) belongs to the \(H^1\)-closure of \(\M_P\), but not to
\(\M_P\) itself, for any finite \(P\ge2\). In the elliptic example below the
exact PDE solution is chosen precisely on this boundary. The reduced loss can
therefore be driven down to the exact energy level, but no finite parameter
vector realizes the limiting state.

Let \(e_1=(1,0,\ldots,0)\) and, for \(h>0\), set
\begin{equation}
\label{eq:collision-sequence}
q_h(x):=\frac{G(x+he_1)-G(x-he_1)}{2h}.
\end{equation}
Then \(q_h\) is a two-neuron realization with nonzero weights
\(\pm(2h)^{-1}\), distinct centers \(\mp he_1\), and barycenter zero.
In one space dimension this collision has the simple state-space picture shown
in \Cref{fig:state-space-collision}: the parameter weights blow up, but the
realized functions approach the smooth derivative \(G'\).

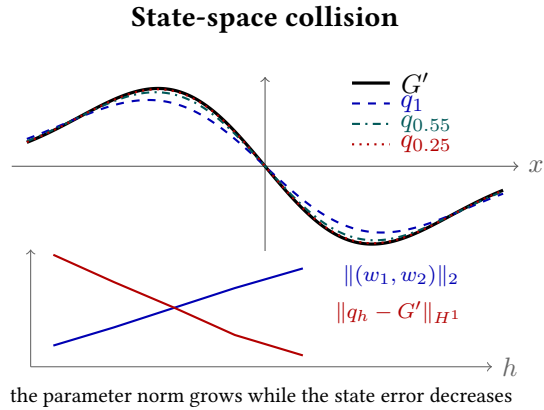
\begin{figure}[ht]
\centering
\begin{tikzpicture}[x=1cm,y=1cm,font=\small]
  \node[font=\bfseries] at (0,1.95) {State-space collision};

  \begin{scope}
    \draw[->,black!65] (-3.35,0) -- (3.35,0) node[right] {\(x\)};
    \draw[->,black!65] (0,-1.12) -- (0,1.18);
    \draw[very thick,black,domain=-3.15:3.15,samples=150,smooth,variable=\x]
      plot ({\x},{8.5*(-\x/2)*(1/sqrt(4*pi))*exp(-\x*\x/4)});
    \draw[thick,blue!70!black,dashed,domain=-3.15:3.15,samples=150,smooth,variable=\x]
      plot ({\x},{8.5*((1/sqrt(4*pi))*exp(-(\x+1.0)*(\x+1.0)/4)
        -(1/sqrt(4*pi))*exp(-(\x-1.0)*(\x-1.0)/4))/(2*1.0)});
    \draw[thick,teal!70!black,dash dot,domain=-3.15:3.15,samples=150,smooth,variable=\x]
      plot ({\x},{8.5*((1/sqrt(4*pi))*exp(-(\x+0.55)*(\x+0.55)/4)
        -(1/sqrt(4*pi))*exp(-(\x-0.55)*(\x-0.55)/4))/(2*0.55)});
    \draw[thick,red!70!black,dotted,domain=-3.15:3.15,samples=150,smooth,variable=\x]
      plot ({\x},{8.5*((1/sqrt(4*pi))*exp(-(\x+0.25)*(\x+0.25)/4)
        -(1/sqrt(4*pi))*exp(-(\x-0.25)*(\x-0.25)/4))/(2*0.25)});
    \draw[black,very thick] (1.15,1.10) -- (1.65,1.10)
      node[right] {\(G'\)};
    \draw[blue!70!black,thick,dashed] (1.15,0.83) -- (1.65,0.83)
      node[right] {\(q_{1}\)};
    \draw[teal!70!black,thick,dash dot] (1.15,0.56) -- (1.65,0.56)
      node[right] {\(q_{0.55}\)};
    \draw[red!70!black,thick,dotted] (1.15,0.29) -- (1.65,0.29)
      node[right] {\(q_{0.25}\)};
  \end{scope}

  \begin{scope}[yshift=-2.65cm,xshift=-0.05cm]
    \draw[->,black!60] (-3.05,0) -- (3.05,0) node[right] {\(h\)};
    \draw[->,black!60] (-3.05,0) -- (-3.05,1.55);
    \draw[blue!70!black,thick] (-2.75,0.28) -- (-1.95,0.52) -- (-1.15,0.78)
      -- (-0.35,1.04) -- (0.55,1.30);
    \draw[red!70!black,thick] (-2.75,1.48) -- (-1.95,1.12) -- (-1.15,0.78)
      -- (-0.35,0.42) -- (0.55,0.15);
    \node[font=\scriptsize,blue!70!black] at (1.85,1.20) {\(\|(w_1,w_2)\|_2\)};
    \node[font=\scriptsize,red!70!black] at (1.82,0.72) {\(\|q_h-G'\|_{H^1}\)};
    \node[font=\scriptsize] at (0,-0.40) {the parameter norm grows while the state error decreases};
  \end{scope}
\end{tikzpicture}
\caption{The explicit collision in state space for \(d=1\). The two-neuron
difference quotient \(q_h=(G(\cdot+h)-G(\cdot-h))/(2h)\) approaches the smooth
state \(G'\) as the centers collapse and the weights grow like \(h^{-1}\). The
lower panel makes the same point quantitatively: the parameter norm increases
while the state error decreases.}
\label{fig:state-space-collision}
\end{figure}

\begin{proposition}[Nondegenerate collision]
\label{prop:nonclosed}
As \(h\downarrow0\),
\[
q_h\longrightarrow \partial_1G
\qquad\text{strongly in }H^1(\R^d),
\]
whereas \(\partial_1G\notin\M_P\) for every finite \(P\). Consequently,
\(\M_P\) is not closed in \(H^1(\R^d)\) for any \(P\ge2\).
The parameter norm of \eqref{eq:collision-sequence} diverges even after
quotienting by permutations, and the sequence uses neither inactive nor
repeated neurons.
\end{proposition}

\begin{proof}
The convergence is the centered finite-difference formula in \(H^1\); it also
follows by dominated convergence after Fourier transformation. If
\(\partial_1G\in\M_P\), division of the Fourier identity by the nowhere-zero
\(\widehat G\) would give
\[
i\xi_1=\sum_{j=1}^P w_je^{-ix_j\cdot\xi}.
\]
On the line \(\xi=te_1\), the right-hand side is bounded for real \(t\), while
the left-hand side is unbounded. This contradiction proves
\(\partial_1G\notin\M_P\). The two weights have Euclidean norm
\((\sqrt2h)^{-1}\), which diverges and is unchanged by permutation.
\end{proof}

The collision also produces nonattainment.

\begin{theorem}[Nonattainment and parameter escape]
\label{thm:nonattainment}
Let
\[
u=\partial_1G,
\qquad
f=(-\Delta+I)\partial_1G.
\]
Then \(f\) is a Schwartz function and, for every \(P\ge2\),
\[
I_P=I,
\qquad
\operatorname*{argmin}_{v\in\M_P}J(v)=\varnothing.
\]
Every parameter-space minimizing sequence satisfies
\(\|\Theta^k\|\to\infty\). For \(P=2\), the explicit sequence in
\eqref{eq:collision-sequence} is nondegenerate in the sense described above.
\end{theorem}

\begin{proof}
By \Cref{prop:nonclosed} and the continuity of \(J\),
\(J(q_h)\to J(u)=I\), so \(I_P=I\) for every \(P\ge2\). If the infimum were
attained by \(v\in\M_P\), \eqref{eq:energy-identity} would imply \(v=u\),
contradicting \(u\notin\M_P\).

Parameter escape follows from \Cref{prop:abstract-parameter-escape}: here
\(\Theta_P=\R^{P(d+1)}\), the realization map \(\Phi_P\) is continuous into
\(H^1\), and \(J\) is continuous.
\end{proof}

\begin{remark}[Boundary search at fixed width]
\label{rem:boundary-search}
\Cref{thm:nonattainment} should be read as a boundary-search statement. For
each finite width \(P\ge2\), the target state \(u=\partial_1G\) lies on the
\(H^1\)-boundary of the nonclosed realization class \(\M_P\), but not in
\(\M_P\). Increasing \(P\) by any finite amount does not remove this obstruction:
the same state is outside every finite class \(\M_P\), while the two-neuron
collision sequence, padded with zero weights, belongs to every larger class.
Thus the reduced infimum is the exact energy value for every finite \(P\), but
it is never attained.

Consequently, an unregularized parameter search that keeps trying to minimize
the exact reduced loss has no finite optimizer to find. In this variational
sense it is forced to chase the boundary, with centers colliding and weights
blowing up. This does not mean that a computation should follow the path to
arbitrarily large parameters. On the contrary, the state error is already small
for moderately small \(h\), and an effective implementation would stop once the
state, residual, or energy tolerance is reached, or would regularize the
parameters. What fails is the naive expectation that exact minimization at fixed
width should converge to a bounded parameter vector.
\end{remark}

\begin{remark}[Interior and boundary targets]
\label{rem:interior-boundary-targets}
The example is positional. It should not be read as saying that finite
Gaussian ansatzes must always exhibit parameter escape. If the exact state
already belongs to \(\M_P\), then the reduced minimum is attained and no
boundary chase is forced. For instance, a single fixed-variance Gaussian
translate is represented exactly with \(P=1\), and a finite linear combination
of \(N\) such translates is represented exactly for every \(P\ge N\), after
padding by zero weights. In those cases the elliptic energy has an exact
realized minimizer, although the parametrization may still be nonunique.

The divergence in \Cref{thm:nonattainment} occurs because the chosen PDE
solution lies in the \(H^1\)-closure of every finite realization class
\(\M_P\), but not in any \(\M_P\) itself. The ansatz has enough expressive
power to approach the state arbitrarily well, but only by moving toward a
boundary point with no finite parameter representative. This is the usual role
of such examples in numerical analysis: they identify a mechanism by which a
method can fail at the parameter level, not a claim that every target state or
every computation must diverge.
\end{remark}

The distinction can be summarized as follows: an interior target may be reached
by bounded parameters, while a boundary target may be approached in state space
only through parameters that escape.

\subsection*{A two-neuron computation}

For \(d=1\), Plancherel's identity gives the exact error formula
\begin{equation}
\label{eq:quadrature-formula}
\|q_h-G'\|_{H^1(\R)}^2
=\frac{1}{2\pi}\int_{\R}(1+\xi^2)
\left(\frac{\sin(h\xi)}{h}-\xi\right)^2e^{-2\xi^2}\dd\xi.
\end{equation}
For the right-hand side in \Cref{thm:nonattainment}, the energy gap is one half
of this quantity. \Cref{tab:collision} reports direct quadrature of
\eqref{eq:quadrature-formula}. The table isolates the coercivity gap from any
training dynamics: the weight norm grows like \(h^{-1}\), whereas the state
error is \(O(h^2)\) and the energy gap is \(O(h^4)\).

\begin{table}[ht]
\centering
\caption{Two active neurons approaching \(G'\) in one dimension.}
\label{tab:collision}
\begin{tabular}{@{}rrrr@{}}
\toprule
\(h\) & \(\|(w_1,w_2)\|_2\) & \(\|q_h-G'\|_{H^1}\)
& \(J(q_h)-J(G')\) \\
\midrule
0.50000 &  1.414214 & \(1.45609\times10^{-2}\) & \(1.06010\times10^{-4}\) \\
0.25000 &  2.828427 & \(3.71097\times10^{-3}\) & \(6.88567\times10^{-6}\) \\
0.12500 &  5.656854 & \(9.32243\times10^{-4}\) & \(4.34539\times10^{-7}\) \\
0.06250 & 11.313708 & \(2.33343\times10^{-4}\) & \(2.72246\times10^{-8}\) \\
0.03125 & 22.627417 & \(5.83535\times10^{-5}\) & \(1.70257\times10^{-9}\) \\
\bottomrule
\end{tabular}
\end{table}

\subsection*{Regularization along the collision path}

The same example quantifies the bias--stability tradeoff created by a
total-variation penalty. In one dimension, the Fourier representation gives
\[
\widehat q_h(\xi)
=i\frac{\sin(h\xi)}{h}e^{-\xi^2},
\qquad
\widehat{G'}(\xi)=i\xi e^{-\xi^2}.
\]
Taylor expansion under the integrable weight in
\eqref{eq:quadrature-formula} yields
\begin{equation}
\label{eq:collision-asymptotic}
\|q_h-G'\|_{H^1}^2
=C_Gh^4+O(h^6),
\qquad h\downarrow0,
\end{equation}
where
\begin{equation}
\label{eq:collision-constant}
C_G:=\frac{1}{72\pi}
\int_{\R}(1+\xi^2)\xi^6e^{-2\xi^2}\dd\xi>0.
\end{equation}
Indeed,
\(\sin(h\xi)/h-\xi=-h^2\xi^3/6+O(h^4\xi^5)\), and dominated convergence
applies after squaring.

For the representing measure
\(\mu_h=(\delta_{-h}-\delta_h)/(2h)\), one has
\(\|\mu_h\|_{\TV}=h^{-1}\). Along this collision path, the regularized
functional \eqref{eq:tikhonov}, introduced below, therefore satisfies
\begin{equation}
\label{eq:regularized-collision-balance}
J_\delta(\mu_h)-J(G')
=\frac{C_G}{2}h^4+\delta h^{-2}+O(h^6).
\end{equation}
Balancing the two leading terms gives
\begin{equation}
\label{eq:regularized-collision-scales}
h_\delta\asymp\delta^{1/6},
\qquad
\|\mu_{h_\delta}\|_{\TV}\asymp\delta^{-1/6},
\qquad
\|q_{h_\delta}-G'\|_{H^1}\asymp\delta^{1/3}.
\end{equation}
Thus every fixed \(\delta>0\) prevents the singular limit, but the parameter
norm again diverges as the regularization is removed. These powers describe
the explicit collision family; they are not asserted to be optimal rates for
the global regularized problem.

\section{State convergence and quantitative Gaussian approximation}
\label{sec:state-convergence}
\label{sec:rates}

\subsection{State convergence}

Define the best-approximation error
\begin{equation}
\label{eq:deltaP}
\delta_P:=\inf_{v\in\M_P}\|u-v\|_{H^1(\R^d)}.
\end{equation}
In the notation of \Cref{subsec:gaussian-dictionary}, \(\alpha_P=I_P-I\). The
energy identity \eqref{eq:energy-identity} converts the abstract estimate
\eqref{eq:abstract-state-error-budget} into a sharp best-approximation bound
for this model.

\begin{theorem}[State convergence and optimization error]
\label{thm:state-convergence}
Let \(f\in H^{-1}(\R^d)\), and let \(u\) solve \eqref{eq:elliptic}.

\begin{enumerate}[label=(\alph*),leftmargin=*]
\item For fixed \(P\), let \((v_P^k)_k\subset\M_P\) satisfy
\(J(v_P^k)\to I_P\), and let \(u_P\) be any weak \(H^1\)-accumulation point.
Then
\begin{equation}
\label{eq:fixed-width-bound}
\|u_P-u\|_{H^1}\le\delta_P.
\end{equation}
In particular, every such selection satisfies \(u_P\to u\) strongly in
\(H^1(\R^d)\) as \(P\to\infty\).

\item If \(v_P\in\M_P\) is a diagonal family of near-minimizers satisfying
\[
J(v_P)\le I_P+\varepsilon_P,
\qquad \varepsilon_P\ge0,
\]
then
\begin{equation}
\label{eq:optimization-error}
\|v_P-u\|_{H^1}^2\le\delta_P^2+2\varepsilon_P.
\end{equation}
Thus \(v_P\to u\) strongly whenever \(\delta_P\to0\) and
\(\varepsilon_P\to0\).
\end{enumerate}
\end{theorem}

\begin{proof}
Coercivity makes \((v_P^k)_k\) bounded. Weak lower semicontinuity and
\eqref{eq:energy-identity} give
\[
\frac12\|u_P-u\|_{H^1}^2
=J(u_P)-I
\le I_P-I
\le\frac12\delta_P^2,
\]
which proves \eqref{eq:fixed-width-bound}. Wiener's theorem gives
\(\delta_P\to0\). For part (b),
\[
\frac12\|v_P-u\|_{H^1}^2
=J(v_P)-I
\le I_P-I+\varepsilon_P
\le\frac12\delta_P^2+\varepsilon_P.
\]
\end{proof}

\begin{remark}
At fixed width, \(u_P\) may lie in
\(\overline{\M_P}^{\,H^1}\setminus\M_P\). It is a relaxed physical state,
not necessarily the realization of a finite parameter vector. The diagonal
estimate \eqref{eq:optimization-error} avoids selecting weak limits and is the
more direct statement for computed near-minimizers.
\end{remark}

\subsection{The weighted spectral estimate}

We next bound \(\delta_P\) explicitly. The proof combines Hermite spectral
truncation, finite-difference realization of the resulting Gaussian jet, and a
localization step. 

Set
\[
K(x):=e^{|x|^2/4}
\]
and, for \(m\in\N\), define
\[
H^m(K):=
\left\{v\in H^m(\R^d):
\sum_{|\alpha|\le m}\int_{\R^d}|D^\alpha v|^2K\dd x<\infty\right\}.
\]
Consider the self-adjoint operator in \(L^2(K\dd x)\)
\begin{equation}
\label{eq:weighted-operator}
\mathcal Lv=-\Delta v-\frac{x\cdot\nabla v}{2}
=-K^{-1}\operatorname{div}(K\nabla v).
\end{equation}
Its eigenspaces are spanned by derivatives of
\(e^{-|x|^2/4}=(4\pi)^{d/2}G\), and the eigenvalue corresponding to total
derivative order \(n\) is \((d+n)/2\). These eigenfunctions form an orthogonal
basis in the relevant weighted Sobolev scales; see \cite{zuazua2020}.

Let \(\Pi_L\) denote the spectral projection onto total derivative order at
most \(L\). The standard spectral tail estimate is
\begin{equation}
\label{eq:spectral-tail}
\|v-\Pi_Lv\|_{H^1(K)}
\le C_d(L+1)^{-1/2}\|v\|_{H^2(K)}.
\end{equation}
For completeness, if \(v=\sum_{n\ge0}\pi_n v\) is the orthogonal spectral
expansion and \(\mu_n=(d+n)/2\), the quadratic-form identity for
\(\mathcal L\) and the weighted elliptic estimate give
\[
\|v\|_{H^1(K)}^2\asymp_d
\sum_{n\ge0}(1+\mu_n)\|\pi_n v\|_{L^2(K)}^2,
\qquad
\sum_{n\ge0}(1+\mu_n)^2\|\pi_n v\|_{L^2(K)}^2
\le C_d\|v\|_{H^2(K)}^2.
\]
Restricting the first sum to \(n>L\) and using
\((1+\mu_n)\le(1+\mu_{L+1})^{-1}(1+\mu_n)^2\) proves
\eqref{eq:spectral-tail}. This also makes explicit why its exponent is
\(-1/2\).
The range of \(\Pi_L\) is
\[
V_L:=\operatorname{span}\{D^\alpha G:|\alpha|\le L\},
\qquad
N_L:=\dim V_L=\binom{L+d}{d}.
\]

It remains to represent Gaussian derivatives by a fixed number of Gaussian
translates.

\begin{lemma}[Realization of a Gaussian jet by colliding translates]
\label{lem:gaussian-jet}
For every \(L\ge0\),
\[
V_L\subset\overline{\M_{N_L}}^{\,H^1(\R^d)}.
\]
More precisely, for every \(q\in V_L\), there are \(N_L\) fixed, distinct
vectors \(z_1,\ldots,z_{N_L}\) and coefficients \(c_j(h)\) such that
\[
q_h(x):=\sum_{j=1}^{N_L}c_j(h)G(x-hz_j)
\longrightarrow q
\quad\text{in }H^1(\R^d)
\quad\text{as }h\downarrow0.
\]
The coefficients may grow as \(O(h^{-L})\).
\end{lemma}

\begin{proof}
Choose \(N_L\) points \(z_j\in\R^d\) that are unisolvent for polynomials of
total degree at most \(L\). Thus the multivariate Vandermonde matrix
\((z_j^\alpha)_{|\alpha|\le L,\,1\le j\le N_L}\) is invertible. Write
\(q=\sum_{|\alpha|\le L}a_\alpha D^\alpha G\) and choose the coefficients
\(c_j(h)\) to solve the moment equations
\begin{equation}
\label{eq:moment-equations}
\sum_{j=1}^{N_L}c_j(h)
\frac{(-h)^{|\alpha|}z_j^\alpha}{\alpha!}=a_\alpha,
\qquad |\alpha|\le L.
\end{equation}
The inverse Vandermonde matrix shows that
\(\max_j|c_j(h)|\le C(q,L)h^{-L}\) for \(0<h\le1\).

Taylor's formula for translations, taken in \(H^1\), gives
\[
G(\cdot-hz_j)
=\sum_{|\alpha|\le L}
\frac{(-h)^{|\alpha|}z_j^\alpha}{\alpha!}D^\alpha G
+r_{j,L}(h),
\qquad
\|r_{j,L}(h)\|_{H^1}\le C(L,z_j)h^{L+1},
\]
because \(G\in H^m\) for every \(m\). The moment equations cancel the Taylor
polynomial against \(q\), and hence
\[
\|q_h-q\|_{H^1}
\le\sum_{j=1}^{N_L}|c_j(h)|\,
\|r_{j,L}(h)\|_{H^1}
\le C(q,L)h\longrightarrow0.
\]
\end{proof}

The coefficient growth in \Cref{lem:gaussian-jet} is not hidden: it is another
instance of parameter escape. The lemma is an approximation statement, not a
claim of numerical conditioning.

\begin{proposition}[Weighted approximation rate]
\label{prop:weighted-rate}
There is \(C_d>0\) such that, for \(v\in H^2(K)\) and \(P\ge2\),
\begin{equation}
\label{eq:weighted-P-rate}
\inf_{w\in\M_P}\|v-w\|_{H^1(\R^d)}
\le C_d P^{-1/(2d)}\|v\|_{H^2(K)}.
\end{equation}
\end{proposition}

\begin{proof}
Since \(K\ge1\), the unweighted \(H^1\)-norm is bounded by the weighted one.
For \(P=N_L\), combine \eqref{eq:spectral-tail} with
\Cref{lem:gaussian-jet} and take the infimum over \(\M_{N_L}\):
\[
\inf_{w\in\M_{N_L}}\|v-w\|_{H^1}
\le C_d(L+1)^{-1/2}\|v\|_{H^2(K)}.
\]
For general \(P\), take the largest \(L\) such that \(N_L\le P\), pad with
zero-weight neurons, and use \(N_L\asymp_d(L+1)^d\).
\end{proof}

\subsection{Localization to the whole space}

Write
\[
L^2_1(\R^d)
:=L^2(\R^d;(1+|x|^2)\dd x),
\qquad
\|f\|_{L^2_1}^2
:=\int_{\R^d}(1+|x|^2)|f(x)|^2\dd x,
\]
and
\[
H^1_1(\R^d)
:=\{v\in H^1(\R^d):\, v,\nabla v\in L^2_1(\R^d)\},
\qquad
\|v\|_{H^1_1}^2
:=\int_{\R^d}(1+|x|^2)(|v|^2+|\nabla v|^2)\dd x.
\]
\begin{theorem}[Constructive whole-space rate]
\label{thm:log-rate}
There is a constant \(C_d>0\) such that, for
\(u\in H^1_1(\R^d)\cap H^2(\R^d)\) and \(P\ge2\),
\begin{equation}
\label{eq:log-rate}
\inf_{w\in\M_P}\|u-w\|_{H^1(\R^d)}
\le
\frac{C_d}{\sqrt{\log(2+P)}}
\bigl(\|u\|_{H^1_1}+\|u\|_{H^2}\bigr).
\end{equation}
Consequently, the relaxed states and near-minimizers in
\Cref{thm:state-convergence} satisfy the same state rate, with the additional
optimization term shown in \eqref{eq:optimization-error}.
\end{theorem}

\begin{proof}
Choose \(\varphi\in C_c^\infty(\R^d)\) such that \(\varphi=1\) on \(B_1\)
and \(\varphi=0\) outside \(B_2\). For \(R\ge1\), set
\(\varphi_R(x)=\varphi(x/R)\) and \(u_R=u\varphi_R\). Direct differentiation
and the weighted tail estimate give
\begin{equation}
\label{eq:cutoff-error}
\|u-u_R\|_{H^1}\le \frac{C}{R}\|u\|_{H^1_1}.
\end{equation}
The product rule gives \(\|u_R\|_{H^2}\le C\|u\|_{H^2}\), uniformly for
\(R\ge1\). Since \(\supp u_R\subset B_{2R}\) and
\(K(x)\le e^{R^2}\) on that ball,
\begin{equation}
\label{eq:weighted-cutoff-norm}
\|u_R\|_{H^2(K)}\le Ce^{R^2/2}\|u\|_{H^2}.
\end{equation}
By \Cref{prop:weighted-rate},
\[
\inf_{w\in\M_P}\|u-w\|_{H^1}
\le \frac{C}{R}\|u\|_{H^1_1}
+Ce^{R^2/2}P^{-1/(2d)}\|u\|_{H^2}.
\]
For sufficiently large \(P\), take
\(R^2=(2d)^{-1}\log P\). The second term is then
\(CP^{-1/(4d)}\|u\|_{H^2}\), which is bounded by the right-hand side of
\eqref{eq:log-rate}. Enlarging the constant covers the remaining finite range
of \(P\).
\end{proof}

The regularity hypothesis follows from a natural weighted assumption on the
source.

\begin{lemma}[Weighted elliptic estimate]
\label{lem:weighted-elliptic}
If
\[
f\in L^2_1(\R^d)\]
then the solution of \eqref{eq:elliptic} belongs to
\(H^1_1(\R^d)\cap H^2(\R^d)\), and
\begin{equation}
\label{eq:weighted-elliptic-estimate}
\|u\|_{H^1_1}+\|u\|_{H^2}
\le C_d\|f\|_{L^2_1}.
\end{equation}
\end{lemma}

\begin{proof}
Let \(m(\xi)=(1+|\xi|^2)^{-1}\). Then
\(\widehat u=m\widehat f\), and the unweighted \(H^2\)-estimate follows because
\((1+|\xi|^2)m\) is bounded. The assumption \(f\in L^2_1\) is equivalent,
by Plancherel, to \(\widehat f\in H^1(\R^d)\). The multipliers
\(m\), \(\xi_jm\), and all of their first derivatives are bounded. Therefore
\[
\nabla_\xi(m\widehat f),
\qquad
\nabla_\xi(\xi_jm\widehat f),\quad j=1,\ldots,d,
\]
belong to \(L^2\), with norms bounded by \(C_d\|f\|_{L^2_1}\). A second
application of Plancherel identifies these quantities with \(xu\) and
\(x\,\partial_j u\), respectively, and proves
\eqref{eq:weighted-elliptic-estimate}.
\end{proof}

Combining \Cref{thm:state-convergence,thm:log-rate,lem:weighted-elliptic}
gives the advertised PDE estimate.

\begin{corollary}[Logarithmic state rate]
\label{cor:log-state-rate}
Assume \(f\in L^2_1(\R^d)\). If \(u_P\) is a weak accumulation point of a
fixed-width minimizing sequence as in \Cref{thm:state-convergence}, then
\[
\|u_P-u\|_{H^1}
\le\frac{C_d}{\sqrt{\log(2+P)}}\|f\|_{L^2_1}.
\]
If \(v_P\in\M_P\) satisfies \(J(v_P)\le I_P+\varepsilon_P\), then
\[
\|v_P-u\|_{H^1}
\le\frac{C_d}{\sqrt{\log(2+P)}}\|f\|_{L^2_1}
+\sqrt{2\varepsilon_P}.
\]
\end{corollary}

\begin{remark}[Perspective on sharpness]
The logarithmic estimate is a uniform whole-space rate on bounded subsets of
\(H^1_1(\R^d)\cap H^2(\R^d)\). Three features make this the natural scale at the
level of generality considered here. First, the spectral step is sharp in the
Hermite--Gaussian Hilbert scale: the tail from \(H^2(K)\) to \(H^1(K)\) has
order \((L+1)^{-1/2}\). Second, the realization step does not add an arbitrary
discretization loss: the Hermite--Gaussian jet space
\[
V_L=\operatorname{span}\{D^\alpha G:|\alpha|\le L\}
\]
has dimension \(N_L=\binom{L+d}{d}\), and \Cref{lem:gaussian-jet} realizes an
arbitrary element of this space with \(N_L\) colliding fixed-variance
translates. This is the correct count for the uniform jet-realization mechanism
used in the proof. Third, the passage from the unweighted whole-space setting
to the exponential Hermite scale is made only through one polynomial moment,
which produces the localization balance leading to \((\log P)^{-1/2}\).

This should not be confused with a pointwise optimality statement for each
fixed target. If \(u\in H^1_1(\R^d)\cap H^2(\R^d)\) is fixed, then the
localization tail satisfies \(\|u-u_R\|_{H^1}=o(R^{-1})\). The same construction
therefore gives \(o((\log P)^{-1/2})\) for that particular \(u\). The
logarithmic estimate is instead a uniform statement, with a constant depending
only on the stated norms. We do not prove a matching minimax lower bound for
the nonlinear class \(\M_P\). Such a lower bound would have to exclude all
target-dependent choices of centers and not only the Hermite projection followed
by the uniform jet construction above. At this level of generality, and with
the Gaussian variance fixed, no sharper uniform finite-width whole-space
estimate should be expected without imposing additional decay, analyticity,
variable scales, or a different approximation structure.
\end{remark}

\subsection{Diagonal near-minimizers}

The state estimate also shows how to choose one element from each fixed-width
minimizing sequence.

\begin{proposition}[Selecting sufficiently accurate minimizing states]
\label{prop:diagonal-selection}
For each \(P\), let \((v_P^k)_{k\ge1}\subset\M_P\) satisfy
\(J(v_P^k)\to I_P\) as \(k\to\infty\).

\begin{enumerate}[label=(\alph*),leftmargin=*]
\item There is a diagonal choice \(k(P)\) such that
\begin{equation}
\label{eq:diagonal-energy}
J(v_P^{k(P)})\le I_P+\frac{1}{\log(2+P)}.
\end{equation}
If \(f\in L^2_1(\R^d)\), then
\begin{equation}
\label{eq:diagonal-log-rate}
\|v_P^{k(P)}-u\|_{H^1}
\le
\frac{C_d\bigl(1+\|f\|_{L^2_1}\bigr)}{\sqrt{\log(2+P)}}.
\end{equation}

\item More generally, if \(z_P\in\M_P\) satisfies
\[
J(z_P)\le I_P+\varepsilon+o(1)
\]
for a fixed \(\varepsilon\ge0\), then
\begin{equation}
\label{eq:fixed-optimization-tolerance}
\limsup_{P\to\infty}\|z_P-u\|_{H^1}
\le\sqrt{2\varepsilon}.
\end{equation}
\end{enumerate}
\end{proposition}

\begin{proof}
For each \(P\), convergence to \(I_P\) permits a choice satisfying
\eqref{eq:diagonal-energy}. Estimates \eqref{eq:diagonal-log-rate} and
\eqref{eq:fixed-optimization-tolerance} follow from
\eqref{eq:optimization-error}, \(\delta_P\to0\), and, in the first case,
\Cref{thm:log-rate,lem:weighted-elliptic} above.
\end{proof}

This selection is existential; it gives no bound on \(k(P)\). Any
width--iteration complexity estimate would require a separate analysis of the
chosen training dynamics.

\section{Measure-valued relaxation and regularization}
\label{sec:relaxation}

\subsection{A measure-valued relaxation of the finite-neuron problem}
\label{subsec:measure-relaxation}

We now ask whether the loss of compactness observed for finite Gaussian
mixtures disappears after passing to a measure-valued, convex relaxation of the
problem. The idea is to replace the finite collection of weights and centers by
a finite signed Radon measure. In this formulation, finite Gaussian mixtures
are recovered by atomic measures, while general measures give a larger
admissible class.

Let \(\M(\R^d)\) denote the space of finite signed Radon measures, equipped
with the total-variation norm. For \(\mu\in\M(\R^d)\), define
\[
v_\mu:=G*\mu,
\qquad
J_r(\mu):=J(v_\mu),
\qquad
I_r:=\inf_{\mu\in\M(\R^d)}J_r(\mu).
\]
Young's inequality for measures gives \(v_\mu\in H^1(\R^d)\). If
\[
\mu=\sum_{j=1}^P w_j\delta_{x_j},
\]
then
\[
G*\mu=\sum_{j=1}^P w_jG(\cdot-x_j),
\]
so the finite realization class \(\M_P\) corresponds to atomic measures with
at most \(P\) atoms. Thus the problem defining \(I_r\) is a relaxed version of
the finite-neuron problem.

The next result shows that the relaxation has the correct limiting state, but
does not restore coercivity in the representing measures.

\begin{proposition}[Relaxed minimizing states]
\label{prop:relaxed-states}
One has \(I_r=I\). If \((\mu_k)_k\) is any minimizing sequence for \(I_r\),
then
\[
G*\mu_k\longrightarrow u
\qquad\text{strongly in }H^1(\R^d).
\]
Moreover, there exist minimizing sequences whose total variation tends to
infinity. Thus the relaxed energy does not, by itself, provide a total-variation
bound on representing measures.
\end{proposition}

\begin{proof}
Clearly \(I\le I_r\), since \(I\) is the minimum of \(J\) over all of
\(H^1(\R^d)\). Conversely, atomic measures generate the dense union
\(\bigcup_P\M_P\), and therefore \(I_r\le I\). Hence \(I_r=I\).

If \((\mu_k)_k\) is minimizing, then
\[
J(G*\mu_k)\to J(u).
\]
Using the energy identity \eqref{eq:energy-identity}, we obtain
\[
\frac12\|G*\mu_k-u\|_{H^1}^2
=
J(G*\mu_k)-J(u)\to0,
\]
which proves the strong convergence of the realized states.

For the example \(u=\partial_1G\), the measures
\[
\mu_h:=\frac{\delta_{-he_1}-\delta_{he_1}}{2h}
\]
satisfy
\[
G*\mu_h=q_h\to\partial_1G
\qquad\text{strongly in }H^1(\R^d),
\]
whereas
\[
\|\mu_h\|_{\TV}=h^{-1}\to\infty.
\]
These measures are therefore a minimizing sequence for \(I_r=I\), but they are
unbounded in total variation.
\end{proof}

Thus the relaxation reproduces the limiting state, but the representing
measures may still escape in total variation.

\subsection{Fourier interpretation of relaxed noncoercivity}

The preceding example reflects a simple Fourier mechanism: convolution with
the Gaussian smooths the measure before the energy sees it. Indeed,
\begin{equation}
\label{eq:relaxed-fourier-state}
\widehat{v_\mu}(\xi)
=e^{-|\xi|^2}\widehat\mu(\xi),
\qquad
\widehat{\nabla v_\mu}(\xi)
=i\xi e^{-|\xi|^2}\widehat\mu(\xi).
\end{equation}
Consequently, an \(H^1\)-bound on \(v_\mu\) controls only
\begin{equation}
\label{eq:relaxed-fourier-control}
(1+|\xi|^2)^{1/2}e^{-|\xi|^2}\widehat\mu(\xi)
\quad\text{in }L^2(\R^d).
\end{equation}
Because of the exponentially decaying factor, this estimate gives no bound on
\(\|\mu\|_{\TV}\). In the collision example,
\(\widehat\mu_h(\xi)=i\sin(h\xi_1)/h\) converges pointwise to
\(i\xi_1\), while \(\|\mu_h\|_{\TV}\to\infty\). The state is therefore
regularized by convolution, not by compactness of the representing measures.

Convexity of \(J_r\) in the measure variable removes the finite-particle center
nonconvexity at the formal infinite-dimensional level. However, minimizing
sequences may still leave every bounded total-variation set. Particle
discretizations of the relaxed problem can therefore reproduce the same
concentration mechanism as the original Gaussian parametrization; compare
\cite{bach2017,chizat2018,savarese2019}.

A total-variation penalty restores compactness at the measure level. For
\(\delta>0\), define
\begin{equation}
\label{eq:tikhonov}
J_\delta(\mu):=J(G*\mu)+\delta\|\mu\|_{\TV}^2.
\end{equation}

\begin{proposition}[Existence for the regularized relaxed problem]
\label{prop:tikhonov-existence}
For every \(\delta>0\), \(J_\delta\) attains its minimum on
\(\M(\R^d)\).
\end{proposition}

\begin{proof}
Since \(J\ge I\), every sublevel of \(J_\delta\) is bounded in total
variation. Banach--Alaoglu and the separability of \(C_0(\R^d)\) give
sequential weak-star compactness. If
\(\mu_k\rightharpoonup^*\mu\), then
\(G*\mu_k\rightharpoonup G*\mu\) in \(H^1\): testing against an
\(H^1\)-function produces a continuous function of the translation variable
that vanishes at infinity. The weak lower semicontinuity of \(J\) and of the
total-variation norm completes the direct-method argument.
\end{proof}

The discrete counterpart adds
\[
\delta\Bigl(\sum_j |w_j|\Bigr)^2
\]
to the reduced loss. This controls colliding weights, but it does not control
centers of inactive neurons, remove permutation symmetries, or make the center
optimization convex. Moreover, fixed \(\delta>0\) biases the state, while the
limit \(\delta\downarrow0\) may recover parameter escape.


\section{Kernel choice, activation functions, and persistence of the gap}
\label{sec:alternative-kernels}

The Gaussian is exceptionally smoothing because its Fourier transform decays
exponentially. A slower multiplier might improve the relation between parameter
and state norms, but only if the associated convolution can be inverted stably
in norms compatible with the PDE energy. We do not attempt a full
characterization here.

One formal alternative is the Riesz kernel
\begin{equation}
\label{eq:riesz-kernel}
H_\alpha(x)=|x|^{-\alpha},
\qquad 0<\alpha<d,
\end{equation}
understood as a tempered distribution. Its Fourier transform is
\[
\widehat H_\alpha(\xi)=C_{d,\alpha}|\xi|^{\alpha-d}.
\]
It lies at the weak-Lebesgue endpoint
\(H_\alpha\in L^{d/\alpha,\infty}(\R^d)\), rather than in \(L^1\), so all
convolution statements require the fractional-integration interpretation.
If \(v_\mu=H_\alpha*\mu\), then formally
\begin{equation}
\label{eq:riesz-gradient}
|\widehat{\nabla v_\mu}(\xi)|
=C_{d,\alpha}|\xi|^{\alpha-d+1}|\widehat\mu(\xi)|.
\end{equation}
When \(d\ge2\), the choice \(\alpha=d-1\) makes the multiplier in
\eqref{eq:riesz-gradient} of order zero:
\[
\widehat{\nabla v_\mu}(\xi)
=C_d i\frac{\xi}{|\xi|}\widehat\mu(\xi).
\]
Thus, for an \(L^2\)-density \(\mu\), the homogeneous energy
\(\|\nabla v_\mu\|_{L^2}\) is comparable to \(\|\mu\|_{L^2}\) by the
boundedness of the Riesz transforms. This is much less smoothing than the
Gaussian map.

It does not, however, yield a coercive parametrization of the full
\(H^1(\R^d)\)-norm. The \(L^2\)-part of the state requires
\[
|\xi|^{-1}\widehat\mu(\xi)\in L^2(\R^d),
\]
which creates a low-frequency obstruction; finite measures need not possess
\(L^2\)-densities; and \(H_{d-1}\) is not an \(L^1\)-activation. Riesz kernels
therefore trade Gaussian high-frequency smoothing for low-frequency and domain
difficulties. They are an instructive comparison, not a ready-made remedy.

The collision mechanism is not restricted to Gaussians. If a smooth activation
\(\sigma\) and its derivative belong to the chosen state space, then
\[
\frac{\sigma(\,\cdot+h)-\sigma(\,\cdot-h)}{2h}
\longrightarrow\sigma'
\]
in that space under the corresponding translation-continuity assumptions. If
\(\sigma'\) is not contained in the fixed realization class, the class is again
nonclosed. 

For deep architectures the activation, architecture, domain, and topology have
to be checked separately. The Gaussian construction should therefore be read as
a model mechanism, not as a theorem for all networks. In a given deep
architecture, an analogous argument would require showing that a colliding
configuration in one layer produces a limiting feature, such as a derivative or
first-order variation of an activation profile, and that the subsequent layers
propagate this limit continuously. The relevant continuity, stability, and
nondegeneracy properties are architecture-dependent.

Likewise, variable-coefficient elliptic equations, nonlinear variational
problems, bounded domains, and evolution equations can retain the coercivity
gap whenever their state functional is coercive but their realization map has
noncompact fibers or a nonclosed image. The abstract theorem in
\Cref{sec:abstract} distinguishes this PDE stability statement from the
architecture-specific realization analysis.

\section{Residual minimization and state-space hybrid modelling}
\label{sec:residual-hybrid}

\subsection{Residual minimization and PINNs}

The coercivity gap is not specific to energy minimization. It also appears in
residual-based methods, including Physics-Informed Neural Networks (PINNs) and
variational PINN formulations \cite{rojas2024}. In that setting, the loss may
control the realized state through a PDE stability estimate, while giving no
coercive control of the neural parameters. The relevant question for the
physical approximation is therefore the stability of the residual norm, not
compactness of a particular parameter representation.

For the model operator
\[
A=-\Delta+I:H^1(\R^d)\to H^{-1}(\R^d),
\]
endow \(H^{-1}(\R^d)\) with the dual norm induced by the \(H^1\)-energy inner
product. Then \(A\) is the Riesz isometry. Hence, if \(u\) solves \(Au=f\),
then
\begin{equation}
\label{eq:residual-identity}
\mathcal R(v):=\|Av-f\|_{H^{-1}}^2
=\|v-u\|_{H^1}^2.
\end{equation}
Thus residual minimization in the \(H^{-1}\)-stability norm is exactly
state-error minimization in \(H^1\). In particular, if
\[
\mathcal R(v_P)\le \inf_{w\in\M_P}\mathcal R(w)+\eta_P,
\]
then
\begin{equation}
\label{eq:residual-state-bound}
\|v_P-u\|_{H^1}^2
\le
\inf_{w\in\M_P}\|w-u\|_{H^1}^2+\eta_P.
\end{equation}
Consequently, the approximation results obtained above for the Gaussian
realization classes immediately transfer to this residual formulation: small
residual implies convergence of the realized states, but not boundedness of
the representing parameters.

This is the residual form of the same separation. The PDE stability norm
controls the physical error \(v_P-u\), whereas the neural parametrization may
still approach collision strata, exhibit weight blow-up, or otherwise lose
compactness.

The choice of residual norm is essential. A pointwise, empirical, or
collocation PINN loss is not automatically equivalent to the continuous
\(H^{-1}\)-stability norm. Such a loss yields state convergence only when it is
connected to the continuous residual by an additional sampling, quadrature, or
discrete stability estimate. Similarly, an \(L^2\)-residual controls a stronger
regularity framework: if \(f\in L^2(\R^d)\), then
\(A:H^2(\R^d)\to L^2(\R^d)\) is an isomorphism and
\[
\|v-u\|_{H^2}
\le C\|Av-f\|_{L^2}.
\]
Here the trial functions must belong to the \(L^2\)-operator domain, and the
regularity estimate must hold for the equation and domain under consideration.
These distinctions do not remove the parameter--state separation: even when the
residual guarantees convergence of the computed states in \(H^2\), it need not
yield compactness of the corresponding parameters.
\subsection{A state-space HYCO formulation for hybrid physics--data modelling}

We have so far considered two purely physics-based variational principles:
energy minimization and residual minimization. In many computational settings,
however, the loss also contains a data-misfit term, for instance when the PDE
model is combined with a finite number of measurements. We refer to such
functionals as hybrid physics--data losses.

The HYCO viewpoint \cite{hyco2025,hyco2026} is useful independently of the
parameter-coercivity issue: the physical model, the learned source, the
synthetic state, and the observations are coupled at the state level. The
purpose of this subsection is to formulate this coupling in Hilbert norms
compatible with PDE stability. The coercivity-gap discussion then adds a
separate warning: stability of the hybrid state problem does not by itself
imply compactness of the finite parameters used to represent its components.

Point evaluation is not continuous on \(H^1(\R^d)\) for \(d\ge2\). Thus,
when the hybrid loss is formulated at the energy level, point measurements
should be replaced by bounded observation functionals, for instance local
averages,
\[
\ell_i(v)=\int_{\R^d}\rho_i v\,\dd x,
\qquad \rho_i\in L^2(\R^d),
\]
or, more generally,
\[
\ell_i(v)=\ip{\rho_i}{v}_{H^{-1},H^1},
\qquad \rho_i\in H^{-1}(\R^d).
\]
This is also closer to physical measurements, which usually average the state
over a nonzero sensor volume. Genuine point observations are admissible only
when the state space has enough regularity to embed into \(C^0\). In particular,
in residual minimization measured in \(L^2\), where the natural state regularity
is \(H^2\), point observations are continuous in dimensions \(d\le3\).

To recover the concrete finite-particle representation used in hybrid
physics--data methods, one may employ independent Gaussian mixtures for the
source, physical state, and synthetic state:
\begin{equation}
\label{eq:hybrid-gaussian-components}
\begin{aligned}
f_{*,P}(x)
&=\sum_{j=1}^P f_jG(x-x_{f,j}),\\
u_{{\rm phy},P}(x)
&=\sum_{j=1}^P p_jG(x-x_{p,j}),\\
u_{{\rm syn},P}(x)
&=\sum_{j=1}^P s_jG(x-x_{s,j}).
\end{aligned}
\end{equation}
The corresponding parameter vector is
\[
\Theta_P=
\bigl((f_j,x_{f,j}),(p_j,x_{p,j}),(s_j,x_{s,j})\bigr)_{j=1}^P.
\]
Using a common width is only a notational convenience; the three widths may be
chosen independently. The loss below is evaluated on the realizations in
\eqref{eq:hybrid-gaussian-components}, while the observations remain bounded
functionals rather than ill-defined point values in \(H^1\).

For data \(y_i\in\R\), \(\nu>0\), and \(\lambda>0\), consider
\begin{equation}
\label{eq:hybrid-loss}
\begin{aligned}
\mathscr L(f_*,u_{\rm phy},u_{\rm syn})
:={}&
\frac12\|Au_{\rm phy}-f_*\|_{H^{-1}}^2
+\frac{\nu}{2}\|f_*\|_{H^{-1}}^2 \\
&+\frac{1}{2N}\sum_{i=1}^N
|\ell_i(u_{\rm syn})-y_i|^2
+\frac{\lambda}{2}\|u_{\rm phy}-u_{\rm syn}\|_{H^1}^2 .
\end{aligned}
\end{equation}

\begin{proposition}[Well-posed hybrid state problem]
\label{prop:hybrid}
The functional \(\mathscr L\) is continuous, coercive, and strongly convex on
\[
H^{-1}(\R^d)\times H^1(\R^d)\times H^1(\R^d).
\]
It therefore has a unique minimizer. If all three components are approximated
by increasing finite Gaussian spans and the computed triples are globally
near-minimizing with optimization error tending to zero, then their realized
states converge strongly to this minimizer in the product space.
\end{proposition}

\begin{proof}
Let \(z=(f_*,u_{\rm phy},u_{\rm syn})\). The positive quadratic part of
\(\mathscr L\), without the observation term, is
\[
Q(z):=\frac12\|Au_{\rm phy}-f_*\|_{H^{-1}}^2
+\frac{\nu}{2}\|f_*\|_{H^{-1}}^2
+\frac{\lambda}{2}\|u_{\rm phy}-u_{\rm syn}\|_{H^1}^2.
\]
It controls the product norm. Indeed, \(f_*\) is controlled by the
\(\nu\)-term, and the Riesz-isometry property of \(A\) gives
\[
\|u_{\rm phy}\|_{H^1}
\le \|Au_{\rm phy}-f_*\|_{H^{-1}}
+\|f_*\|_{H^{-1}},
\]
while
\[
\|u_{\rm syn}\|_{H^1}
\le \|u_{\rm syn}-u_{\rm phy}\|_{H^1}
+\|u_{\rm phy}\|_{H^1}.
\]
Consequently there is \(c_{\nu,\lambda}>0\) such that
\[
Q(z)\ge c_{\nu,\lambda}
\bigl(
\|f_*\|_{H^{-1}}^2+\|u_{\rm phy}\|_{H^1}^2
+\|u_{\rm syn}\|_{H^1}^2
\bigr).
\]
The observation term is continuous and convex because the \(\ell_i\) are
bounded. Hence \(\mathscr L\) is coercive and strongly convex on the product
space. Existence and uniqueness follow by the direct method. Gaussian spans are
dense in \(H^1\), and consequently also dense in \(H^{-1}\); the convergence
assertion is
\Cref{thm:abstract-state-stability} applied in the product space.
\end{proof}

The proposition is stated for the global minimization problem, but the same
state-space principle is the one used in HYCO-type formulations, once combined
with the descent or convergence properties of the corresponding alternating
updates. Thus the hybrid model is not introduced here as a remedy for parameter
noncoercivity. It is a stable state-level way of coupling physics, data, and
learned components. The additional message of the present paper is that this
stability controls the realized triple \((f_*,u_{\rm phy},u_{\rm syn})\), but
not necessarily the finite-dimensional parameterization \(\Theta_P\). Strong
convergence of the realized hybrid states can therefore coexist with escape of
the parameters representing them.

\section{Scope, limitations, and open directions}

The results of this paper isolate a general mechanism: variational neural PDE
solvers may be coercive in the realized state while remaining non-coercive in
the parameters. This coercivity gap is not specific to the Gaussian model used
above. The present construction provides a transparent PDE setting in which the
phenomenon can be proved completely, but the same state-versus-parameter
distinction is intrinsic to many neural and particle-based approximation
schemes. The following directions indicate where the mechanism should be made
explicit in further settings.

\begin{itemize}[leftmargin=*]
\item \textbf{Architecture.}
The proofs have been written for a shallow, fixed-variance Gaussian radial-basis
class in order to display the mechanism in its simplest form. Smooth
activations can generate analogous difference-quotient collisions, and
fixed-size realization sets are often nonclosed \cite{petersen2021}. For each
concrete deep architecture, however, the corresponding statement has to be
proved from the realization map and from the noncompact parameter directions
that leave the realized state convergent.

\item \textbf{Optimization.}
The state estimates are formulated for globally near-minimizing sequences. This
is the natural variational level at which the coercivity gap appears. Algorithmic
training adds a second layer: one has to quantify how gradient, alternating, or
regularized dynamics approach such near-minimizers. This does not change the
state-versus-parameter mechanism; it determines which minimizing or
near-minimizing sequences are actually produced by a given optimization method.

\item \textbf{Approximation rates.}
The estimates obtained here are natural, and essentially sharp, when viewed
from the Gaussian--Hermite spectral construction on which the proof is based.
Their deterioration in the Sobolev setting comes from the passage from this
spectral framework to estimates involving only one moment of the target
function. Whether this loss is intrinsic to the approximation class, or can be
reduced by a more refined argument, remains an open question.

Faster rates require additional structure on the target, for instance stronger
decay, analyticity, adaptive scaling properties, or membership in
dimension-robust approximation spaces such as spectral Barron classes
\cite{chen2026barron}.

\item \textbf{Qualitative PDE structure.}
The coercivity gap also interacts with qualitative PDE properties. For
\(f\ge0\), the exact solution of \eqref{eq:elliptic} is nonnegative, but an
unrestricted signed Gaussian mixture need not be. Moreover,
\(v\in\M_P\) does not imply \(v^+\in\M_P\), so the standard variational argument
based on replacing a competitor by its positive part cannot be performed inside
\(\M_P\). This shows that preserving comparison principles, positivity, or
maximum principles requires realization classes designed to encode these
structures while retaining approximation power.

\item \textbf{Numerical design.}
The phenomenon identified here provides a criterion for numerical design:
methods should monitor state accuracy first, in norms dictated by PDE
stability. Parameter norms and conditioning should also be recorded, but they
should be interpreted as diagnostics of the representation and the optimizer,
not as substitutes for the physical error. Freezing or regularizing centers
converts part of the problem into a stable linear approximation, but it also
reduces adaptivity. Frozen centers may be drawn as random features or selected
through quantization and Voronoi procedures such as Lloyd's algorithm
\cite{lloyd1982}. A systematic comparison of adaptive centers, random features,
and deterministic radial-basis designs should therefore track state errors,
residuals, parameter norms, conditioning, and computational cost.

\item \textbf{Renormalization of escaping parameters.}
When parameter escape occurs, it may still be possible to extract useful
information from the escaping configuration by a suitable renormalization, in
the spirit of blow-up analysis for PDEs. For Gaussian classes, collisions of
centers and diverging amplitudes can produce derivative-type limiting profiles
through difference-quotient mechanisms. Such a procedure would not remove the
coercivity gap, since the original parameters remain noncompact, but it could
help identify the compactified representation underlying the convergent state.
Developing numerical renormalization strategies for such escaping minimizing
sequences remains an open direction.

\item \textbf{Hybrid methods.}
HYCO-type losses should also be viewed as hybrid modelling tools in their own
right. The well-posed state functional in \Cref{prop:hybrid} shows that the
realized source, physical state, and synthetic state can be coupled and
controlled in the natural product space. The coercivity-gap analysis is a
separate qualification: this state control does not make the underlying
finite-dimensional parameters coercive. The remaining work is to write the
corresponding algorithmic analysis for alternating HYCO-type updates, including
the role of coupling and regularization parameters, concentration in the three
realization classes, and noisy or incomplete observations.
\end{itemize}
\section{Conclusion}

This paper isolates a coercivity gap that is easy to miss in neural PDE
methods. A coercive PDE energy controls the realized physical state, but it need
not control the parameters used to represent that state. In the Gaussian
elliptic model this separation is complete: two active neurons collide, the
reduced minimum is not attained, and every exact minimizing parameter sequence
escapes, while the corresponding realized states converge strongly in the
energy norm.

The example should be read as a structural diagnosis rather than as a universal
divergence statement. If the exact solution already belongs to a finite
Gaussian class, as for a single Gaussian or a finite Gaussian mixture with
sufficient width, no boundary chase is forced. Parameter escape occurs here
because the chosen solution lies in the \(H^1\)-closure of each finite
realization class but outside the class itself. The ansatz is therefore rich
enough to approximate the state arbitrarily well, but only through parameter
sequences that move toward a boundary point with no finite representative.

The broader methodological message is that neural PDE solvers should be
analyzed first as state-space approximation methods and only then as
finite-dimensional parameter optimizations. The parameters are the coordinates
used by the algorithm; the PDE sees the realized state. In nonclosed
realization classes these two levels may separate: the state may converge in
the coercive PDE norm while every exact minimizing parameter sequence escapes.
This does not invalidate neural PDE solvers. It explains why they can perform
well at the level of physical fields even when the associated parameter
dynamics are unstable, nonunique, or poorly conditioned.

For computations, the first validation criterion should therefore be the state
error when it is available, and otherwise computable proxies such as energy
gaps, stability-norm residuals, or observation errors. Parameter convergence,
boundedness, and conditioning remain important, but they are secondary
diagnostics rather than the primary notion of PDE accuracy. Residual, relaxed,
and hybrid formulations fit the same viewpoint when their norms are compatible
with PDE stability. Parameter regularization can restore compactness, but it
also introduces bias and leaves the nonconvex training problem in place. The
lesson is both cautionary and positive: parameter compactness should not be
assumed, but state convergence can still be proved and should be the central
object of analysis.

\section*{Acknowledgments}

The author thanks Antonio \'Alvarez-L\'opez (Universidad Aut\'onoma de Madrid)
and Daniel Fern\'andez (Friedrich--Alexander--Universit\"at
Erlangen--N\"urnberg) for fruitful discussions.

The author was funded by the Alexander von Humboldt Professorship program, the
ERC Advanced Grant CoDeFeL, Grant PID2023-146872OB-I00-DyCMaMod of MICIU
(Spain), COST Actions CA24136 (InterCoML) and CA24122 (mSPACE), supported by
COST (European Cooperation in Science and Technology), project AFOSR 24IOE027,
and SURE-AI Centre grant 357482 of the Research Council of Norway.

\appendix

\section{Technical complements for the weighted Gaussian approximation}
\label{app:weighted-complements}

This appendix records the weighted spectral and localization details that
underlie \Cref{sec:rates}. The same self-similar operator appears in the
large-time analysis of the heat equation and in moment methods for inverse
problems; see \cite{zuazua2020,liu2026moments}.

\subsection{The self-similar elliptic operator}

Recall
\[
K(x)=e^{|x|^2/4},
\qquad
\mathcal L=-\Delta-\frac{x\cdot\nabla}{2}
=-K^{-1}\operatorname{div}(K\nabla).
\]
The weighted scalar product is
\[
(v,w)_K:=\int_{\R^d}v(x)w(x)K(x)\dd x.
\]

\begin{proposition}[Weighted spectral framework]
\label{prop:appendix-weighted-framework}
The operator \(\mathcal L\) has the following properties.

\begin{enumerate}[label=(\roman*),leftmargin=*]
\item The embedding \(H^1(K)\hookrightarrow L^2(K)\) is compact.

\item The form
\(a(v,w)=\int_{\R^d}\nabla v\cdot\nabla w\,K\dd x\)
defines a positive self-adjoint realization of \(\mathcal L\) with compact
resolvent. Equivalently, \(\mathcal L:H^1(K)\to H^1(K)'\) is an isomorphism
when the spaces are equipped with the equivalent form norms.

\item The eigenvalues, indexed by total derivative order \(n\ge0\), are
\[
\mu_n=\frac{d+n}{2},
\]
and the associated eigenspace is
\[
E_n=\operatorname{span}\{D^\alpha e^{-|x|^2/4}:|\alpha|=n\}.
\]

\item After normalization, these eigenfunctions form an orthonormal basis of
\(L^2(K)\) and an orthogonal basis in the corresponding weighted Sobolev
scales.
\end{enumerate}
\end{proposition}

\begin{proof}
The divergence identity gives symmetry and the form relation
\[
(\mathcal Lv,v)_K=\int_{\R^d}|\nabla v|^2K\dd x.
\]
The weighted Poincar\'e inequality and the growth of \(K\) yield coercivity and
compactness. Direct differentiation shows that the Gaussian derivatives are
eigenfunctions with the displayed eigenvalues. Completeness follows after
conjugating \(\mathcal L\) by \(K^{1/2}\), which reduces it to a shifted
harmonic oscillator. These standard facts, including the equivalence of the
integer spectral and derivative norms, are detailed in \cite{zuazua2020}.
\end{proof}

If \(v=\sum_{n\ge0}\pi_n v\) is the resulting spectral expansion and
\(v_L=\sum_{n=0}^L\pi_n v\), then
\begin{align*}
\|v-v_L\|_{H^1(K)}^2
&\le C_d\sum_{n>L}(1+\mu_n)\|\pi_n v\|_{L^2(K)}^2\\
&\le \frac{C_d}{1+\mu_{L+1}}
\sum_{n>L}(1+\mu_n)^2\|\pi_n v\|_{L^2(K)}^2\\
&\le \frac{C_d}{L+1}\|v\|_{H^2(K)}^2.
\end{align*}
Taking square roots gives precisely the \((L+1)^{-1/2}\) norm rate used in
\eqref{eq:spectral-tail}. The square root is essential; omitting it would lead
to the incorrect exponent \(P^{-1/d}\).

\subsection{Detailed localization estimates}

Let \(\varphi\in C_c^\infty(\R^d)\) satisfy
\[
0\le\varphi\le1,
\qquad
\varphi=1\text{ on }B_1,
\qquad
\varphi=0\text{ outside }B_2,
\]
and set \(\varphi_R(x)=\varphi(x/R)\), \(u_R=u\varphi_R\). Then
\[
u-u_R=(1-\varphi_R)u,
\qquad
\nabla(u-u_R)
=(1-\varphi_R)\nabla u-u\nabla\varphi_R.
\]
Since \(\nabla\varphi_R\) is supported in
\(B_{2R}\setminus B_R\) and
\(\|\nabla\varphi_R\|_\infty\le C/R\),
\begin{align*}
\|u-u_R\|_{H^1}
\le{}&
\|u\|_{L^2(\R^d\setminus B_R)}
+\|\nabla u\|_{L^2(\R^d\setminus B_R)}\\
&+\frac{C}{R}\|u\|_{L^2(B_{2R}\setminus B_R)}.
\end{align*}
If \(u\in H^1_1\), multiplication by \(|x|\ge R\) on the exterior region
gives
\[
\|u-u_R\|_{H^1}\le \frac{C}{R}\|u\|_{H^1_1}.
\]

For \(|\alpha|\le2\), the product rule gives
\[
D^\alpha u_R
=\sum_{\beta\le\alpha}
\binom{\alpha}{\beta}D^\beta u\,
D^{\alpha-\beta}\varphi_R,
\qquad
\|D^\gamma\varphi_R\|_\infty\le C_\gamma R^{-|\gamma|}.
\]
Consequently, for \(R\ge1\),
\begin{align*}
\|u_R\|_{H^2}
&\le C\bigl(\|u\|_{H^2}
+R^{-1}\|u\|_{H^1}+R^{-2}\|u\|_{L^2}\bigr)
\le C\|u\|_{H^2}.
\end{align*}
Finally, \(\supp u_R\subset B_{2R}\) and
\(K(x)\le e^{R^2}\) on this ball, so
\[
\|u_R\|_{H^2(K)}\le Ce^{R^2/2}\|u\|_{H^2}.
\]
Combining these three estimates with \Cref{prop:weighted-rate} produces
\[
\inf_{w\in\M_P}\|u-w\|_{H^1}
\le \frac{C}{R}\|u\|_{H^1_1}
+Ce^{R^2/2}P^{-1/(2d)}\|u\|_{H^2},
\]
from which \eqref{eq:log-rate} follows by the choice made in
\Cref{thm:log-rate}.

\subsection{An alternative weighted energy argument}

The Fourier proof of \Cref{lem:weighted-elliptic} is concise, but the weighted
\(H^1\)-estimate can also be obtained directly from the PDE\@. This derivation is
useful for variable-coefficient problems where an explicit Fourier multiplier
is unavailable.

Choose bounded smooth weights \(\rho_R\) increasing pointwise to
\(1+|x|^2\) and satisfying, uniformly in \(R\),
\begin{equation}
\label{eq:truncated-weight-properties}
1\le\rho_R\le1+|x|^2,
\qquad
|\nabla\rho_R|^2\le C\rho_R.
\end{equation}
Such weights are obtained by smoothly truncating \(1+|x|^2\) at height
\(1+R^2\). Since \(\rho_R\) and its gradient are bounded, \(\rho_Ru\) is an
admissible test function in the weak equation. We obtain
\begin{align}
\label{eq:weighted-energy-identity-appendix}
\int_{\R^d}\rho_R(|\nabla u|^2+|u|^2)\dd x
={}&\int_{\R^d}\rho_Rfu\dd x
-\int_{\R^d}u\nabla u\cdot\nabla\rho_R\dd x.
\end{align}
For arbitrary \(\varepsilon>0\), Young's inequality and
\eqref{eq:truncated-weight-properties} give
\begin{align*}
\left|\int\rho_Rfu\right|
&\le \frac14\int\rho_R|u|^2
+C\int(1+|x|^2)|f|^2,\\
\left|\int u\nabla u\cdot\nabla\rho_R\right|
&\le \varepsilon\int\rho_R|\nabla u|^2
+C_\varepsilon\int|u|^2.
\end{align*}
Choosing \(\varepsilon\) small and using the unweighted energy estimate to
control \(\|u\|_{L^2}\), we find
\begin{equation}
\label{eq:uniform-weighted-energy}
\int_{\R^d}\rho_R(|\nabla u|^2+|u|^2)\dd x
\le C\|f\|_{L^2_1}^2,
\end{equation}
with \(C\) independent of \(R\). Monotone convergence (or Fatou's lemma for a
nonmonotone smooth truncation) yields \(u\in H^1_1\). The unweighted
\(H^2\)-bound still follows from elliptic regularity. This proves
\eqref{eq:weighted-elliptic-estimate} without differentiating the Fourier
transform of \(f\).

\subsection{Explicit derivative stencils and shared neuron counts}

The space of Gaussian derivatives of total order at most \(L\) has dimension
\(N_L=\binom{L+d}{d}\). \Cref{lem:gaussian-jet} uses exactly \(N_L\)
distinct translates to approximate an arbitrary element of this space. The
connection with familiar finite differences is easiest to see in one
dimension. For \(m\ge1\), define
\begin{equation}
\label{eq:explicit-centered-stencil}
D_{h,m}G(x)
:=\frac{1}{h^m}\sum_{j=0}^m
(-1)^{m-j}\binom{m}{j}
G\!\left(x+\left(j-\frac m2\right)h\right).
\end{equation}
Then
\begin{equation}
\label{eq:stencil-convergence}
D_{h,m}G\longrightarrow G^{(m)}
\qquad\text{in }H^1(\R)
\quad\text{as }h\downarrow0.
\end{equation}
For example, \(m=1\) is the two-neuron collision used in
\Cref{sec:escape}, while \(m=2\) gives the three-point stencil
\[
\frac{G(x+h)-2G(x)+G(x-h)}{h^2}\longrightarrow G''(x).
\]
To prove \eqref{eq:stencil-convergence}, one may use Taylor's formula in
\(H^1\), or observe that the Fourier multiplier of
\eqref{eq:explicit-centered-stencil} converges to \((i\xi)^m\) and is
dominated by \(|\xi|^m\); the Gaussian supplies every required moment.

For a multi-index \(\alpha\), tensorizing the one-dimensional stencil uses
\(\prod_{r=1}^d(\alpha_r+1)\) translates to approximate \(D^\alpha G\).
Approximating every derivative separately would repeat many centers and would
not give the sharp count \(N_L\). The unisolvent construction in
\Cref{lem:gaussian-jet} instead selects \(N_L\) nodes once and solves the
moment system for the entire polynomial jet. It thereby shares the same
colliding centers among all derivatives of total order at most \(L\).

For fixed \(L\), the finite-difference step \(h\) may be chosen after the
spectral projection so that the jet-realization error is arbitrarily small.
The coefficients can nevertheless grow like \(h^{-L}\), and the inverse
Vandermonde constant depends on the selected nodes. Thus the construction
proves a best-approximation bound with a controlled neuron count, but not a
uniformly conditioned algorithm or a width--iteration complexity bound. This
is precisely the distinction between expressive approximation and stable
parameter recovery emphasized throughout the paper.

\end{document}